\documentclass{amsart}

\usepackage{amsmath}        
\usepackage{amsfonts}       
\usepackage{amsthm}         
\usepackage{bm}             
\usepackage{graphicx}       
\usepackage{psfrag}         
\usepackage{fancyvrb}       
\usepackage{bbding}         
\usepackage{dcolumn}        
\usepackage{booktabs}       
\usepackage{paralist}       
\usepackage{indentfirst}    

\usepackage{amssymb}
\usepackage{fancyhdr}
\usepackage{enumerate}
\usepackage{mathtools}
\usepackage{mathrsfs}

\usepackage{tikz-cd}
\usetikzlibrary{cd}
\usetikzlibrary{babel}
\usepackage{adjustbox}

\usepackage{stackrel}

\usepackage{bm}             
\usepackage{graphicx}       
\usepackage{psfrag}         
\usepackage{fancyvrb}       

\usepackage{bbding}         

\usepackage{icomma}         
\usepackage{dcolumn}        
\usepackage{booktabs}       
\usepackage{paralist}       

\usepackage{enumerate} 
\usepackage{mathtools}
\usepackage[czech,english]{babel}
\usepackage{array}

\usepackage[unicode]{hyperref}
\hypersetup{pdftitle=Cotilting Sheaves on Noetherian Schemes, 
            pdfauthor=Pavel \v{C}oupek and Jan \v{S}\v{t}ov\'{\i}\v{c}ek,
            ps2pdf,
            colorlinks=false,
            urlcolor=blue,
            pdfstartview=FitH,
            pdfpagemode=UseOutlines,
            pdfnewwindow,
            breaklinks,
            hypertexnames=false}

\newtheoremstyle{vetalike}
{6pt}
{3pt}
{\slshape}
{}
{\bfseries}
{.}
{1em}
{}

\newtheoremstyle{deflike}
{6pt}
{6pt}
{}
{}
{\bfseries}
{.}
{1em}
{}

\newtheoremstyle{prlike}
{6pt}
{6pt}
{}
{}
{\bfseries}
{.}
{1em}
{}

\newtheoremstyle{poznslike}
{6pt}
{6pt}
{}
{}
{\bfseries}
{.}
{\newline}
{}

\theoremstyle{vetalike}
\newtheorem{thm}{Theorem}[section]
\newtheorem{prop}[thm]{Proposition}
\newtheorem{lem}[thm]{Lemma}
\newtheorem{cor}[thm]{Corollary}

\theoremstyle{deflike}
\newtheorem{deff}[thm]{Definition}

\newtheorem{rem}[thm]{Remark}
\newtheorem{constr}[thm]{Construction}

\theoremstyle{poznslike}

\theoremstyle{prlike}
\newtheorem{example}[thm]{Example}

\makeatletter
\def\blfootnote{\xdef\@thefnmark{}\@footnotetext}
\makeatother

\newcommand{\card}[1]{\left\lvert{#1}\right\rvert}

\newcommand{\CM}{Cohen-Macaulay }

\renewcommand{\ker}{\mathrm{Ker}\,}
\renewcommand{\Im}{\mathrm{Im}\,}
\newcommand{\coker}{\mathrm{Coker}\,}
\newcommand{\im}{\mathrm{Im}\,}

\newcommand{\ass}{\mathrm{Ass}\,}
\newcommand{\supp}{\mathrm{Supp}\,}
\newcommand{\ann}{\mathrm{Ann}}

\newcommand{\qcoh}{\mathsf{QCoh}}
\newcommand{\coh}{\mathsf{Coh}}
\newcommand{\spec}[1]{\mathrm{Spec}\,#1}
\newcommand{\proj}[1]{\mathrm{Proj}\,#1}
\newcommand{\Mod}{\mathsf{Mod\textnormal{\textsf{-}}}}
\newcommand{\modl}{\mathsf{\textnormal{\textsf{-}}mod}}

\newcommand{\Ab}{\mathsf{Ab}}
\newcommand{\Der}{\mathsf{D}}

\newcommand{\oh}{\mathcal{O}}

\newcommand{\Prodcl}{\mathrm{Prod}}
\newcommand{\Limcl}{\underrightarrow{\mathrm{Lim}}\,}

\newcommand{\injdim}{\mathrm{injdim}\,}

\newcommand{\depth}{\mathrm{depth}\,}
\newcommand{\height}{\mathrm{ht}\,}

\newcommand{\cogen}{\mathrm{Cogen}}

\renewcommand{\hom}{\mathrm{Hom}}
\newcommand{\homSh}{\mathscr{H}\!om}   
\newcommand{\ext}{\mathrm{Ext}}
\newcommand{\ehom}{\mathrm{End}}

\begin{document}

\title{Cotilting sheaves on Noetherian schemes}

\author{Pavel \v{C}oupek}
\address{Purdue University, Department of Mathematics, 150 N. University Street,  West Lafayette, IN 47907, USA}

\author{Jan \v{S}\v{t}ov\'{\i}\v{c}ek}
\address{Charles University, Faculty of Mathematics and Physics, Department of Algebra, Sokolovsk\'{a} 83, 186 75 Prague 8, Czech Republic}
\maketitle


\blfootnote{The first named author was partially supported by the institutional grant SVV 260456 of the Charles University and partially supported by the Ross Fellowship of Purdue University.}

\blfootnote{The second named author was supported by Neuron Fund for Support of Science.}

\begin{abstract}
We develop theory of (possibly large) cotilting objects of injective dimension at most one in general Grothendieck categories. We show that such cotilting objects are always pure-injective and that they characterize the situation where the Grothendieck category is tilted using a torsion pair to another Grothendieck category. We prove that for Noetherian schemes with an ample family of line bundles a cotilting class of quasi-coherent sheaves is closed under injective envelopes if and only if it is invariant under twists by line bundles, and that such cotilting classes are parametrized by specialization closed subsets disjoint from the associated points of the scheme. Finally, we compute the cotilting sheaves of the latter type explicitly for curves as products of direct images of indecomposable injective modules or completed canonical modules at stalks.
\end{abstract}

\setcounter{tocdepth}{2} 
\tableofcontents 

\section{Introduction}

Tilting theory is a collection of well established methods for studying equivalences between triangulated categories in homological algebra. Although it has many facets (see~\cite{HandbookTilting}), in its basic form \cite{Happel,Rickard} it struggles to answer the following question: Given two abelian categories $\mathcal{A}, \mathcal{H},$ which may not be equivalent, how can we characterize the situation where their \emph{derived} categories are equivalent,
\[ \Der(\mathcal{H}) \simeq \Der(\mathcal{A})? \]

If $\mathcal{H}$ is a module category, $\mathcal{H} = \Mod{R}$, the answer is well-known. The ring $R$ is transferred by the equivalence to what is called a tilting complex in $\Der(\mathcal{A})$. On the other hand, existence of a tilting complex in $\Der(\mathcal{A})$ whose endomorphism ring is $R$ ensures such a derived equivalence. This is an extremely powerful tool to study representations of groups and quivers, coherent sheaves in commutative or non-commutative geometry, and also in various other situations.

\smallskip

The starting point for this paper is how to detect the case where $\mathcal{H}$ has an injective cogenerator $W$, e.g.\ if $\mathcal{H}$ is a Grothendieck category. In this case, the image $C$ of $W$ in $\Der(\mathcal{A})$ should morally be a cotilting complex. However, we are on a much more experimental ground now with an attempt to define such a complex intrinsically in $\Der(\mathcal{A})$. There are several definitions available in the case where $C$ is required to be an object of $\mathcal{A}$ (see~\cite{Colpi,ColpiTrlifaj2,ColpiHeart,St3,ParraSaorin,FiorotMattielloSaorin,PositselskiSt}), and rather recent research dealing with the case where $C$ is an actual complex, \cite{PsaroudakisVitoria,NicolasSaorinZvonareva}. One of the major problems with manipulating cotilting complexes is that, unlike the ring in its module category, injective cogenerators are often very far from being finitely generated in any reasonable sense.

Our aim here is to understand the situation in detail for the particular case where $\mathcal{A}$ is also a Grothendieck category, and preferably even the category $\qcoh_X$ of quasi-coherent sheaves on a Noetherian scheme $X$. We are inspired by recent progress in understanding cotilting sheaves of affine schemes in~\cite{St1,HrbekSt}.

We also restrict the shape of derived equivalences which we consider. We only focus on derived equivalences coming from `turning around' a torsion pair $(\mathcal{T}, \mathcal{F})$ in $\mathcal{A}$. This is a very general method introduced by Happel, Reiten and Smal\o{} in~\cite{HRS}. The abelian category $\mathcal{H}$ which we obtain on the other end of the derived equivalence comes equipped with a torsion pair $(\mathcal{F}, \mathcal{T})$ and it has very strong homological bonds to $\mathcal{A}$. This effectively means that we restrict ourselves to cotilting objects in $\mathcal{A}$ whose injective dimension is at most one.

\smallskip

Our main result in this direction (Theorem~\ref{thm:maintheorem}), which generalizes~\cite[\S5]{ParraSaorin}, says that for any Grothendieck category $\mathcal{A}$ equipped with a torsion pair $(\mathcal{T}, \mathcal{F})$ such that $\mathcal{F}$ contains a generator, the tilted category $\mathcal{H}$ (in the sense of Happel, Reiten and Smal\o{}) is again a Grothendieck category if and only if $\mathcal{H}$ has an injective cogenerator if and only if $\mathcal{F}$ has an injective cogenerator $C$ as an exact category. We call such objects of $\mathcal{A}$ cotilting.

In order to prove this result, we had to overcome the following problem---all proofs available for module categories use the fact that cotilting modules are pure-injective. As a general Grothendieck category $\mathcal{A}$ need not be locally finitely presentable, there does not seem to be any good definition of a pure-exact structure on $\mathcal{A}$ available. Still it turns out that it make sense to define pure-injective objects in this context and that all cotilting objects as in the last paragraph are pure-injective (Theorem~\ref{thm:CotiltingPureInjective}).

The pure-injectivity has other nice consequences. For example, this allows us to easily describe the cotilting torsion-free classes in $\mathcal{A}$ (Theorem~\ref{thm:CharTiltingClasses}) and, if $\mathcal{A}$ is a locally Noetherian Grothendieck category, cotilting objects are parametrized, up to product equivalence, by torsion pairs in the category $\mathcal{A}_0$ of Noetherian objects of $\mathcal{A}$ (Theorem~\ref{thm:ClassificationViaTP}; see also~\cite{BuanKrause-cotilting,ParraSaorin}).

If $X$ is a Noetherian scheme, this takes us back to our original aim. We can in principle classify all cotilting sheaves in $\qcoh_X$, up to product equivalence, as soon as we understand torsion pairs in $\coh_X$. If $X$ is affine, all torsion pairs in $\coh_X$ are hereditary and given by a subset $Y\subseteq X$ which is closed under specialization~\cite[\S2]{St1}. If $X$ is non-affine, this is no longer the case. Here we single out the torsion pairs in $\coh_X$ which are hereditary (Theorem~\ref{thm:charTPCoh}). It turns out that these are again precisely the ones classified by specialization closed subsets $Y\subseteq X$, or equivalently the ones for which the torsion class is a tensor ideal. If $X$ has an ample family of line bundles, these are also precisely those for which the torsion-free class is closed under twists by the line bundles.
The consequences for cotilting quasi-coherent sheaves are then summarized in Theorem~\ref{thm:classification}. 

Last but not least, we also illustrate the theory throughout the paper in the case of $1$-dimensional Noetherian schemes by providing an explicit computation of the cotilting sheaves associated to a specialization closed subset. We discuss the example of $\qcoh_{\mathbb{P}^1_k}$, quasi-coherent sheaves on the projective line over a field, in particular. It turns out that all the technical nuances and differences from the affine situation already occur there.

\section{Cotilting objects in Grothendieck categories}
\label{sec:cotilting}

The goal of the section is to establish basic theory of cotilting objects of injective dimension at most $1$ for Grothendieck categories. The definition of such a (infinitely generated) cotilting object in a module category is rather standard, see \cite{Colpi,ColpiTrlifaj2}. However, as mentioned in the introduction, extensions of the concept to more general abelian categories or even triangulated categories are still subject to experiments by various authors, see e.g.\ \cite{ParraSaorin,FiorotMattielloSaorin,PsaroudakisVitoria,NicolasSaorinZvonareva} to name a few.

Here we show that basic aspects of cotilting modules from \cite{Colpi,ColpiTrlifaj2} generalize to Grothendieck categories rather easily and that our definition matches perfectly with the one from \cite{ParraSaorin,FiorotMattielloSaorin,NicolasSaorinZvonareva} (and a posteriori with \cite{PsaroudakisVitoria} as well in view of the results in \cite{NicolasSaorinZvonareva}).

At several steps throughout this section, we make use of results on interaction of the $\ext^1$-functor with infinite products in a Grothendieck category, which are stated in Appendix~\ref{sec:ExtProd}.

We start by recalling the definition of a torsion pair in an abelian category. 

\begin{deff}[{\cite{Dickson}}]\label{def:torznipar}
Let $\mathcal{A}$ be an abelian category. A \emph{torsion pair} in $\mathcal{A}$ is a pair $(\mathcal{T}, \mathcal{F})$ of full subcategories of $\mathcal{A}$ such that 
\begin{enumerate}[(1)]
\item\label{TPDef1}{$\hom_{\mathcal{A}}(\mathcal{T}, \mathcal{F})=0$, and }
\item\label{TPDef2}{for every $A \in \mathcal{A},$ there is an exact sequence

\vspace{0.25cm}
\begin{adjustbox}{max totalsize={1\textwidth}{.9\textheight},center}
\begin{tikzcd}
0 \ar[r]& T \ar[r]& A \ar[r] & F \ar[r] & 0 
\end{tikzcd}
\end{adjustbox}
\vspace{0.15cm}

\noindent with $T \in \mathcal{T}$ and $F \in \mathcal{F}$. We call the object $T$ the \emph{torsion part} of $A$, and the object $F$ the \emph{torsion-free part} of $A$.
}
\end{enumerate}
A torsion pair $(\mathcal{T}, \mathcal{F})$ is called \emph{hereditary} if $\mathcal{T}$ is closed under subobjects. 
\end{deff}

\begin{rem}[\cite{Dickson}]\label{rem:TP} ~
\begin{enumerate}
\item\label{TP1}{It follows from the definition that for each torsion pair $\mathcal{F}=\ker \hom_{\mathcal{A}}(\mathcal{T}, -)$ and $\mathcal{T}=\ker \hom_{\mathcal{A}}(-, \mathcal{F})$. In particular, $\mathcal{T}$ is closed under extensions, factors and under all colimits that exist in $\mathcal{A}$ and, dually, $\mathcal{F}$ is closed under extensions, subobjects and limits.
}
\item\label{TP2}{Suppose that $\mathcal{A}$ is a well-powered abelian category where for every direct system $A_i$ of subobjects of a fixed object $A$ (and monomorphisms compatible with those to $A$), $\varinjlim_i A_i$ exists---this is fulfilled in particular when $\mathcal{A}$ is a Grothendieck category (e.g. $\qcoh_X$) or a Noetherian category (e.g. $\coh_X$ for a Noetherian scheme $X$). Assume that $\mathcal{T}\subseteq\mathcal{A}$ is closed under extensions, quotients and under colimits that exist in $\mathcal{A}$, and put $\mathcal{F}=\ker \hom_{\mathcal{A}}(\mathcal{T}, -)$. Then $(\mathcal{T}, \mathcal{F})$ is a torsion pair.

Indeed, consider $A \in \mathcal{A}$, and take $T \stackrel{\alpha}\hookrightarrow A$ the maximum subobject of $A$ with $T \in \mathcal{T}$. Such a $T$ exists, it is obtained as the image of $\varinjlim_i T_i \rightarrow A,$ where $T_i$ varies over (representatives of) all subobjects of $A$ which are in $\mathcal{T}$. Then necessarily $\coker\alpha\in\mathcal{F}$ and we obtain the exact sequence required by Definition~\hyperref[TPDef2]{\ref*{def:torznipar}~(\ref*{TPDef2})}.}
\end{enumerate}
\end{rem}

We will also make use of the following easy observation, whose proof we omit.

\begin{lem}\label{lem:TorsionPairsInclusion}
Suppose $(\mathcal{T}, \mathcal{F})$ and $(\mathcal{T}', \mathcal{F}')$ are two torsion pairs in an abelian category $\mathcal{A}$ such that $\mathcal{T} \subseteq \mathcal{T}'$ and $\mathcal{F}\subseteq \mathcal{F}'$. Then  $(\mathcal{T}, \mathcal{F})=(\mathcal{T}', \mathcal{F}').$
\end{lem}

Given a class of objects $\mathcal{C}$ in a Grothendieck category $\mathcal{A}$, we further denote
\begin{align*}
{^\perp\mathcal{C}} &= \{ F \in \mathcal{A} \mid \ext^1_\mathcal{A}(F, C) = 0 \text{ for each } C\in\mathcal{C} \},\\
\mathcal{C}^\perp &= \{ F \in \mathcal{A} \mid \ext^1_\mathcal{A}(C, F) = 0 \text{ for each } C\in\mathcal{C} \}
\end{align*}
and we denote by $\Prodcl({\mathcal{C}})$ the class of direct summands of arbitrary direct products of objects from $\mathcal{C}$.
If $\mathcal{C}=\{C\}$ is a singleton, we simply write ${^\perp{C}}$, $C^\perp$ and $\Prodcl({C})$, respectively.
Given an object $C \in \mathcal{A}$, we denote by $\cogen({C})$ the class of all objects $F\in\mathcal{A}$ which are cogenerated by $C$ (that is, which admit an embedding of the form $F\hookrightarrow {C}^{\times I}$).

Now we can give the technically least involved definition of a cotilting object in a Grothendieck category $\mathcal{A}$. It generalizes~\cite[Definition 2.6]{ColpiTiltingGrothendieck} for Grothendieck categories which might not have enough projective objects.

\begin{deff}\label{def:cotilting}
An object $C$ in a Grothendieck category $\mathcal{A}$ is \emph{cotilting} if ${^\perp C}=\cogen({C})$ and the class ${^\perp C}$ contains a generator of $\mathcal{A}$.
The class $\mathcal{C} = \cogen({C})$ is called the \emph{cotilting class} associated with $C$.
\end{deff}

\begin{rem}
Unlike in module categories, ${^\perp C}$ is not automatically generating since $\mathcal{A}$ may not have enough projective objects.
\end{rem}

The following shows at once that each cotilting class is a torsion-free class. We state the lemma in greater generality for later use.

\begin{lem}\label{lem:cotilting-TF}
Let $\mathcal{A}$ be a Grothendieck category and $C\in\mathcal{A}$ be an object such that $\cogen({C})\subseteq {^\perp C}$. Then the class $\cogen({C})$ is torsion-free, with the corresponding torsion class given by
\[ \mathcal{T} = \{ T \in \mathcal{A} \mid \hom_\mathcal{A}(T, C) = 0 \}. \]
In particular, the cotilting class is closed under direct sums.
\end{lem}

\begin{proof}
By the very definition, $\cogen({C})$ is closed under taking subobjects and products, and thus also all limits in $\mathcal{A}$. Once we show that $\cogen({C})$ is also closed under extensions, it will be a torsion-free class by an argument formally dual to Remark~\hyperref[TP2]{\ref*{rem:TP}~(\ref*{TP2})}.

To that end, consider a short exact sequence

\vspace{0.25cm}
\begin{adjustbox}{max totalsize={1\textwidth}{.9\textheight},center}
\begin{tikzcd}
0 \ar[r]& {F}' \ar[r, "j"] & {F} \ar[r, "p"] &  {F}'' \ar[r] & 0  
\end{tikzcd}
\end{adjustbox}
\vspace{0.15cm}

\noindent in $\mathcal{A}$, where ${F}', {F}'' \in \cogen({C})$. That is, we have monomorphisms $i', i''$ forming the solid part of the following diagram:

\vspace{0.25cm}
\begin{adjustbox}{max totalsize={1\textwidth}{.9\textheight},center}
\begin{tikzcd}
0 \ar[r]& {F}' \ar[r, "j"] \ar[d, "i'", hook] & {F} \ar[r, "p"] \ar[d, dotted, "({}^{\;\;k}_{i''p})"] &  {F}'' \ar[r] \ar[d, "i''", hook] & 0  \\
0 \ar[r, dotted]& {C}^{\times I} \ar[r, dotted, "\iota_I"] & {C}^{\times(I \cup J)} 
\ar[r, dotted, "\pi_J"] & {C}^{\times J} 
\ar[r, dotted] & 0
\end{tikzcd}
\end{adjustbox}
\vspace{0.15cm}

\noindent
Regarding the dotted part, we take for the lower sequence the split exact sequence with the canonical projections and inclusions. Since $F'' \in \cogen({C}) \subseteq {^\perp C} = {^\perp {C}^{\times(I \cup J)}}$ by Corollary~\ref{cor:BunoProd}, the composition $\iota_I i'$ extends to a morphism $k\colon F \to {C}^{\times(I \cup J)}$, i.e.\ $\iota_I i' = kj$. One readily checks that this choice of $k$ in the matrix at the middle vertical arrow makes the diagram commutative and that, by the Four Lemma, the middle arrow is a monomorphism. Thus, $F\in \cogen({C})$, as required.

Finally, note that the canonical morphism $\bigoplus_{i\in I} F_i \to \prod_{i\in I} F_i$ from a direct sum to a product is injective in any Grothendieck category, since it is the direct limit of the split monomorphisms $\bigoplus_{i\in J} F_i \simeq \prod_{i\in J} F_i \to \prod_{i\in I} F_i$, where $J$ runs over finite subsets of $I$. It follows that any torsion-free class in a Grothendieck category is closed under taking direct sums.
\end{proof}
Now we aim at giving a homological characterization of cotilting objects along the lines of \cite{Colpi,ColpiTrlifaj2}. However, some care is necessary since products need not be exact in $\mathcal{A}$. We first note that the injective dimension of a cotilting object $C$ is at most one. Indeed, this is a consequence of the following proposition applied to $\mathcal{F} = \cogen({C}) = {^\perp{C}}.$

\begin{prop}\label{prop:cotilting-injdim1}
Let $\mathcal{A}$ be a Grothendieck category and $\mathcal{F}$ be a torsion-free class in $\mathcal{A}$ which contains a generator.
Then $\injdim {C}\leq 1$ for any $C\in\mathcal{F}^\perp$.
\end{prop}

\begin{proof}
Let ${G}$ be an object of $\mathcal{A}$. We show that $\ext_{\mathcal{A}}^2({G}, {C})=0$ by showing that every $2$-fold extension of ${G}$ by ${C}$ represents the trivial class of $\ext^2_{\mathcal{A}}({G}, {C})$. 
	
Consider a $2$-fold extension 
	
\vspace{0.25cm}
\begin{adjustbox}{max totalsize={1\textwidth}{.9\textheight},center}
	\begin{tikzcd} 
		\xi: &0 \ar[r]  & {C} \ar[r]& {E}_2 \ar[r]& {E}_1 \ar[r, "\alpha"] & {G} \ar[r] & 0 \;.& 
	\end{tikzcd}
\end{adjustbox}
\vspace{0.15cm}
	
\noindent Since $\mathcal{F}$ contains a generator and is closed under direct sums by Lemma~\ref{lem:cotilting-TF}, there is an object ${F}_1 \in \mathcal{F}$ and an epimorphism $\varepsilon\colon{F}_1 \rightarrow {E}_1$. If we denote $\beta=\alpha \varepsilon$ and consider the pullback of the projection $\pi\colon E_2 \to \ker\alpha$ along the map $\ker\beta\to\ker\alpha$ induced by $\varepsilon$, we obtain the following commutative diagram with exact rows:
	
\vspace{0.25cm}
\begin{adjustbox}{max totalsize={1\textwidth}{.9\textheight},center}
	\begin{tikzcd}[column sep=small, row sep=tiny] 
		\xi:&0 \ar[rr]&&{C} \ar[rr]&& {E}_2 \ar[rr] \ar[dr, two heads, "\pi"'] && {E}_1 \ar[rr, "\alpha"]  && {G} \ar[rr] && 0 \\
		&&&  &&       & \ker\alpha \ar[ru, hook] &&  & &                                                                \\
		\xi':&0 \ar[rr]&&{C} \ar[uu, equal] \ar[rr, "\gamma"]&& {F}_2 \ar[rr]\ar[rd, "\pi''"', two heads] \ar[uu] && {F}_1  \ar[uu, "\varepsilon", two heads] \ar[rr, "\beta"] && {G} \ar[rr] \ar[uu, equal] && 0  \\
		&&&  &&        & \ker\beta \ar[ru, hook]\ar[uu, "\varepsilon'", crossing over, near start] && & &                                                                
	\end{tikzcd}
\end{adjustbox}
\vspace{0.15cm}
	
That is, the $2$-extensions $\xi$ and $\xi'$ represent the same element of $\ext^2_\mathcal{A}(G, C)$. Now it is enough to show that $\xi'$ is trivial by showing that $\gamma$ is split monic. This, however, immediately follows from the fact that $\ker\beta \in \mathcal{F} \subseteq {^\perp C}$ since $\ker\beta$ is a subobject of $F_1$.
\end{proof}

Suppose now that $C\in\mathcal{A}$ is a cotilting object. Then, by Corollary~\ref{cor:BunoProd}, $C^{\times I}$ is also a cotilting object associated with the same cotilting class. In particular, each $C' \in \Prodcl(C)$ is of injective dimension at most one, and as such admits an injective resolution

\vspace{0.25cm}
\begin{adjustbox}{max totalsize={1\textwidth}{.9\textheight},center}
	\begin{tikzcd} 
		0 \ar[r]  & {C}' \ar[r]& {E}^0 \ar[r]& {E}^1 \ar[r] & 0 \;.&
	\end{tikzcd}
\end{adjustbox}
\vspace{0.15cm}

We will show that, on the other hand, each injective object admits a dual resolution in terms of objects of $\Prodcl{(C)}$. In order to do so, we need a lemma, which in essence generalizes \cite[Proposition~1.8]{Colpi} to our abstract setting.

\begin{lem}\label{lem:cogen=copres}
Let ${C}$ be an object of $\mathcal{A}$ satisfying ${^{\perp}}{C}=\cogen({C})$. Assume that ${K} \in \cogen({C})$. Then there is a short exact sequence

\vspace{0.25cm}
\begin{adjustbox}{max totalsize={1\textwidth}{.9\textheight},center}
\begin{tikzcd}
0 \ar[r]& {K} \ar[r]& {C}^{\times X} \ar[r] & {L} \ar[r] & 0 \;, 
\end{tikzcd}
\end{adjustbox}
\vspace{0.15cm}

\noindent where $X$ is a set and ${L} \in \cogen({C})$.
\end{lem}

\begin{proof}
We put $X=\hom_{\mathcal{A}}({K}, {C})$ and consider the diagonal map $\Delta\colon {K} \rightarrow {C}^{\times X}.$ That is, $\Delta$ is given by $\pi_{\chi} \Delta = \chi,\;\; \chi \in X,$
and it is injective since ${K} \in \cogen({C}).$ Thus, we obtain a short exact sequence

\vspace{0.25cm}
\begin{adjustbox}{max totalsize={1\textwidth}{.9\textheight},center}
\begin{tikzcd}
0 \ar[r]& {K} \ar[r, "\Delta"]& {C}^{\times X} \ar[r] & {L} \ar[r] & 0  
\end{tikzcd}
\end{adjustbox}
\vspace{0.15cm}

\noindent
and it remains to check that ${L} \in \cogen({C})={^{\perp}}{C}$.
Applying $\hom_{\mathcal{A}}(-, {C})$, we obtain a long exact sequence

\vspace{0.25cm}
\begin{adjustbox}{max totalsize={1\textwidth}{.9\textheight},center}
\begin{tikzcd}
\cdots\; \hom_{\mathcal{A}}({C}^{\times X}, {C}) \ar[r, "{-\circ\Delta}"] & \hom_{\mathcal{A}}({K}, {C}) \ar[r] &\ext_{\mathcal{A}}^1({L}, {C}) \ar[r] & \ext_{\mathcal{A}}^1({C}^{\times X}, {C}) \;\cdots\,. 
 \end{tikzcd}
\end{adjustbox}
\vspace{0.15cm}

\noindent The map $-\circ \Delta=\hom_{\mathcal{A}}(\Delta, {C})$ is clearly surjective from construction and $\ext^{1}_{\mathcal{A}}({C}^{\times X}, {C})=0$. Thus, $\ext_{\mathcal{A}}^1({L}, {C})=0$, which concludes the proof. 
\end{proof}

\begin{lem} \label{lem:prod-C}
Given a cotilting object $C$ of $\mathcal{A},$ we have that
$$ \cogen({C}) \cap \cogen({C})^\perp=\Prodcl({C}).$$
\end{lem}

\begin{proof}
To show the inclusion `$\supseteq$', note that $\Prodcl({C}) \subseteq \cogen({C})$ and ${C} \in (^\perp {C})^\perp=\cogen({C})^\perp$, using Definition~\ref{def:cotilting}. It follows that $\Prodcl({C}) \subseteq \cogen({C})^\perp$ since $\cogen({C})^\perp$ is closed under direct summands and direct products by Corollary~\ref{prop:ExtTriv}. 

Conversely, given ${P} \in \cogen({C}) \cap \cogen({C})^\perp$, by Lemma~\ref{lem:cogen=copres} we obtain a short exact sequence $0 \to P \to {C}^{\times X} \to L \to 0$  with $L\in\cogen({C})$. Since ${P} \in \cogen({C})^\perp$, the sequence splits, showing that ${P} \in \Prodcl({C})$ and thus finishing the proof.
\end{proof}

\begin{prop} \label{prop:condition-C3}
Let ${C}$ be a $1$-cotilting object of $\mathcal{A}$. Given an injective object ${W}$ (an injective cogenerator for $\mathcal{A}$ in particular), there is a short exact sequence 

\vspace{0.25cm}
\begin{adjustbox}{max totalsize={1\textwidth}{.9\textheight},center}
\begin{tikzcd}
0 \ar[r]& {C}_1 \ar[r]& {C}_0 \ar[r] & {W} \ar[r] & 0  
\end{tikzcd}
\end{adjustbox}
\vspace{0.15cm}

\noindent with ${C}_0, {C}_1 \in \Prodcl({C})$.
\end{prop}

\begin{proof}
The argument is very similar to the one used in \cite{ColpiTrlifaj} for modules, with necessary modifications due to the fact that we do not have exact products and enough projective objects.

Since ${^{\perp}}{C}=\cogen({C})$ is generating, one can consider an epimorphism $e':{F}\twoheadrightarrow {W}$, where ${F} \in \cogen({C})$. By definition of $\cogen({C})$, there is a monomorphism $\iota:{F}\hookrightarrow {C}^{\times I}$ for some set $I$. By injectivity of ${W}$, $e'$ extends along $\iota$ to a morphism $e: {C}^{\times I}\twoheadrightarrow {W}$. Clearly $e$ is an epimorphism as well.

Thus, we have a short exact sequence

\vspace{0.25cm}
\begin{adjustbox}{max totalsize={1\textwidth}{.9\textheight},center} 
\begin{tikzcd}
0 \ar[r] &  {K}  \ar[r, "i"]& {C}^{\times I} \ar[r, "e"] & {W} \ar[r]& 0,
\end{tikzcd}
\end{adjustbox}
\vspace{0.15cm}

 and, obviously, ${K} \in \cogen({C})$. By Lemma~\ref{lem:cogen=copres} there is a short exact sequence 

\vspace{0.25cm}
\begin{adjustbox}{max totalsize={1\textwidth}{.9\textheight},center} 
\begin{tikzcd}
0  \ar[r] & {K} \ar[r, "j"] & {C}^{\times J} \ar[r] & {L} \ar[r]& 0
\end{tikzcd}
\end{adjustbox}
\vspace{0.15cm} 

for some set $J$ and some ${L} \in \cogen({C}).$
Consider the pushout ${P}$ of $i$ and $j$, which gives rise to a commutative diagram

\vspace{0.25cm}
\begin{adjustbox}{max totalsize={1\textwidth}{.9\textheight},center}
\begin{tikzcd}
& 0 \ar[d] & 0 \ar[d] & & \\
0 \ar[r]& {K} \ar[r, "i"] \ar[d, "j"]& {C}^{\times I} \ar[r, "e"] \ar[d] & {W} \ar[r] \ar[d, equal] & 0  \\
0 \ar[r]& {C}^{\times J} \ar[r] \ar[d]& {P} \ar[r] \ar[d] & {W} \ar[r] & 0  \\
& {L} \ar[r, equal] \ar[d] & {L} \ar[d] & & \\
& 0 & 0 & & 
\end{tikzcd}
\end{adjustbox}
\vspace{0.15cm}

\noindent with exact rows and columns. Observe that since ${L}, {C}^{\times I} \in \cogen({C})$ and $\cogen({C})={^\perp}{C}$ is closed under extensions, we have ${P} \in \cogen({C})$. Since, on the other hand, ${C}^{\times J}, W\in\cogen({C})^\perp=({^\perp}{C})^\perp$, we also have ${P} \in \cogen({C})^\perp.$ Thus, ${P} \in \Prodcl({C})$ by Lemma~\ref{lem:prod-C}, and the second row of the pushout diagram is the desired exact sequence from the statement of the proposition.
\end{proof}

Now we can characterize cotilting objects essentially in terms of the conclusions of Propositions~\ref{prop:cotilting-injdim1} and~\ref{prop:condition-C3}, and the fact that $\ext^1_\mathcal{A}({C}^{\times I}, {C}) = 0$ for all sets $I$.

\begin{thm}\label{thm:CotiltingChar}
Let $\mathcal{A}$ be a Grothendieck category and $C$ be an object in $\mathcal{A}$. Then the following conditions are equivalent:
\begin{enumerate}[(1)]
	\item\label{Char1}{$C$ is cotilting.}
	\item\label{Char2}{$C$ satisfies the following three conditions:
	\begin{enumerate}[(C1)]
		\item\label{C1}{$\injdim C \leq 1$.}
		\item\label{C2}{$\ext_\mathcal{A}^{1}(C^{\times I}, C)=0$ for every set $I$.}
		\item\label{C3}{For every injective cogenerator $W$, there is an exact sequence 
						
		\vspace{0.25cm}
		\begin{adjustbox}{max totalsize={1\textwidth}{.9\textheight},center}
			\begin{tikzcd} 
				0 \ar[r] & C_1 \ar[r] & C_0 \ar[r] & W \ar[r] & 0 
			\end{tikzcd}
		\end{adjustbox}
		\vspace{0.15cm}
						
		\noindent where $C_0, C_1 \in \Prodcl(C)$.}
	\end{enumerate}}
\end{enumerate}
\end{thm}
		
\begin{proof}
The implication $\hyperref[Char1]{(\ref*{Char1})} \Rightarrow \hyperref[Char2]{(\ref*{Char2})}$ is clear from Propositions~\ref{prop:cotilting-injdim1} and~\ref{prop:condition-C3}.

Suppose conversely that $C$ satisfies conditions \hyperref[C1]{(C1)}--\hyperref[C3]{(C3)}.
By \hyperref[C2]{(C2)} we have $\Prodcl({C}) \subseteq {^\perp}{C}$ and, by \hyperref[C1]{(C1)}, ${^\perp}{C}$ is closed under subobjects. Thus ${\cogen({C}) \subseteq {^\perp}{C}}$.

In order to prove the inclusion $\cogen({C}) \supseteq {^\perp}{C}$, note that we already know by Lemma~\ref{lem:cotilting-TF} that $\cogen({C})$ is a torsion-free class of a torsion pair in $\mathcal{A}$. Consider any object $A \in {^{\perp}}{C}$ and the short exact sequence

\vspace{0.25cm}
\begin{adjustbox}{max totalsize={1\textwidth}{.9\textheight},center}
\begin{tikzcd}
0 \ar[r]& {T} \ar[r] & {A} \ar[r] &  {F} \ar[r] & 0  
\end{tikzcd}
\end{adjustbox}
\vspace{0.15cm}

\noindent with ${F} \in \cogen({C})$ and ${T} \in \ker \hom_{\mathcal{A}}(-, {C}).$ In order to prove that ${A} \in \cogen({C})$, it suffices to show that ${T}=0$.
To this end, consider an injective cogenerator ${W} \in \mathcal{A}$ and the short exact sequence 

\vspace{0.25cm}
\begin{adjustbox}{max totalsize={1\textwidth}{.9\textheight},center}
\begin{tikzcd}
0 \ar[r]& {C}_1 \ar[r] & {C}_0 \ar[r] &  {W} \ar[r] & 0  
\end{tikzcd}
\end{adjustbox}
\vspace{0.15cm}

\noindent from axiom \hyperref[C3]{(C3)}. Applying $\hom_{\mathcal{A}}({T}, -)$, we obtain an exact sequence

\vspace{0.25cm}
\begin{adjustbox}{max totalsize={1\textwidth}{.9\textheight},center}
\begin{tikzcd}
\cdots \ar[r]& \hom_{\mathcal{A}}({T}, {C}_0) \ar[r] & \hom_{\mathcal{A}}({T}, {W}) \ar[r] &  \ext^1_{\mathcal{A}}({T}, {C}_1) \ar[r] & \cdots\,,  
\end{tikzcd}
\end{adjustbox}
\vspace{0.15cm}

\noindent where $\hom_{\mathcal{A}}({T}, {C}_0)=0$, and by the fact that ${^\perp}{C} = {^\perp}{\Prodcl({C})}$ is closed under subobjects, also $\ext^1_{\mathcal{A}}({T}, {C}_1)=0$. It follows that $\hom_{\mathcal{A}}({T}, {W})=0$ and, since ${W}$ is a cogenerator, also ${T}=0$. This concludes the proof of the inclusion.

It remains to observe that $\cogen({C}) = {^\perp}{C}$ is generating. Consider any object $G \in \mathcal{A}$ and an injective cogenerator $W$ admitting a monomorphism $\iota:G \hookrightarrow W$ (if $U$ is any injective cogenerator, then there is a monomorphism $G \hookrightarrow U^{\times I}$ for some set $I$ and we can take $W=U^{\times I}$). By \hyperref[C3]{(C3)}, we may consider an exact sequence 
			
\vspace{0.25cm}
\begin{adjustbox}{max totalsize={1\textwidth}{.9\textheight},center}
	\begin{tikzcd} 
		0 \ar[r] & C_1 \ar[r] & C_0 \ar[r, "\pi"] & W \ar[r] & 0 
	\end{tikzcd}
\end{adjustbox}	
\vspace{0.15cm}
			
\noindent with $C_0, C_1 \in \Prodcl({C})$, and take the pullback of $\pi$ along $\iota$,

\vspace{0.25cm}
\begin{adjustbox}{max totalsize={1\textwidth}{.9\textheight},center}
	\begin{tikzcd} 
		0 \ar[r] & C_1 \ar[r]  & C_0 \ar[r, "\pi"] & W \ar[r] & 0 \\
		0 \ar[r] & C_1 \ar[r] \ar[u, equal] & P \ar[r, "\pi'"] \ar[u, "\iota'", hook] & G \ar[r] \ar[u, "\iota", hook] & 0\,.
	\end{tikzcd}
\end{adjustbox}
\vspace{0.15cm}

\noindent
This shows that $G$ is an epimorphic image of $P$ and $P \in \cogen({C})$.
\end{proof}

We conclude the section by showing that products of copies of a cotilting object are homologically well-behaved. This implies that the definition of a $1$-cotilting object from \cite[Definition 8]{NicolasSaorinZvonareva} or \cite[Definition 2.1]{FiorotMattielloSaorin} for more general abelian categories specializes precisely to our definition in the case of Grothendieck categories.
Strictly speaking, the latter references give formally dual definitions of so-called tilting objects, but it is remarked in both of the papers that their results apply to the dual concept as well. In particular, we can use abstract results on the existence of derived equivalences from \cite{NicolasSaorinZvonareva,FiorotMattielloSaorin} (this topic will be discussed in Section~\ref{sec:purity} as well).

We in fact prove a more general result which implies that a Grothendieck category may have interesting full subcategories where products are exact (see e.g.\ Example~\ref{example:P1VsKronecker} below).

\begin{prop}\label{prop:ExactProd}
Let $\mathcal{A}$ be a Grothendieck category and $\mathcal{F}\subseteq\mathcal{A}$ torsion-free class which contains a generator. Given any collection of objects $B_i,$ $i \in I,$ from $\mathcal{F}^\perp$ and any $n>0$, the $n$-th right derived functor of product, $\mathbf{R}^n\prod_{i \in I} B_i,$ vanishes.
\end{prop}

\begin{proof}
First note that, by Proposition~\ref{prop:cotilting-injdim1}, we obtain that $\injdim {B_i}\leq 1$ for all $i\in I.$ That is, $\mathbf{R}^n\prod_{i \in I} B_i = 0$ for all $n>1$.

Now fix a generator $G\in \mathcal{F}$ and for each $i\in I$, fix an injective resolution

\vspace{0.25cm}
\begin{adjustbox}{max totalsize={1\textwidth}{.9\textheight},center}
\begin{tikzcd} 
0 \ar[r]  & B_i \ar[r]& E^0_i \ar[r, "\rho_i"]& E^1_i \ar[r] & 0\;.
\end{tikzcd}
\end{adjustbox}
\vspace{0.15cm}

\noindent
Then $\hom_\mathcal{A}(G, \rho_i)$ is surjective by the assumption and so is $\hom_\mathcal{A}(G, \prod_{i\in I}\rho_i) \simeq \prod_{i\in I}\hom_\mathcal{A}(G, \rho_i)$ since products are exact in the category of abelian groups. As $G$ is a generator for $\mathcal{A}$, it follows that $\prod_i \rho_i\colon \prod_i E^0_i \to \prod_i E^1_i$ is an epimorphism and $\mathbf{R}^1\prod_{i \in I} B_i = \coker \prod_i \rho_i = 0.$
\end{proof}

\begin{cor}\label{cor:ExactProdOfC}
Let and $C \in \mathcal{A}$ be a cotilting object. Then a product of copies of $C$ in $\mathcal{A}$ coincides with the corresponding product of copies of $C$ in the derived category $\Der(\mathcal{A})$.
\end{cor}

\begin{proof}
We apply the previous proposition to $\mathcal{F} = \cogen({C})$.
\end{proof}

\section{Pure-injectivity of cotilting objects}
\label{sec:purity}

Here we focus on a more advanced aspect of cotilting objects in Grothendieck categories. A key ingredient of the corresponding theory for modules is that all cotilting modules are pure-injective, \cite{BazzoniPureInj}. We will prove an analogous result for cotilting objects in Grothendieck categories. This will allow us to characterize which torsion-free classes are associated with cotilting objects. As a consequence, we recover a first version of a classification for locally Noetherian Grothendieck categories from~\cite{BuanKrause-cotilting,ParraSaorin} and, in the next section, we prove that a cotilting object in a Grothendieck category induces a derived equivalence to another abelian category which is again a Grothendieck category.

The first obstacle on the way is that there seems to be no consensus on what the definition of a pure-injective object is for a general Grothendieck category (the definition using pure exact sequences does not apply for a Grothendieck category need not be locally finitely presentable). We use as the definition a characterization of pure-injectivity for modules from~\cite[Theorem 7.1~(vi)]{JensenLenzing}.

\begin{deff}\label{def:PureInj}
An object $C$ in a Grothendieck category $\mathcal{A}$ is \emph{pure-injective} if, for each index set $I$, the summing map $\Sigma\colon C^{\oplus I} \to C$ (whose all components $\Sigma \circ \iota_i\colon C \to C,$ $i\in I,$ are the identity maps on $C$) extends to a homomorphism $\overline{\Sigma}\colon C^{\times I} \to C$.
\end{deff}

Although this is a rather elementary intrinsic definition of pure-injectivity, it will be useful to give a characterization in terms of pure-injectivity of certain modules. We will do so via the Gabriel-Popescu theorem, a version of which we first recall.

\begin{prop}[{\cite[\S X.4]{Stenstrom}}] \label{prop:GabrielPopescu}
Let $\mathcal{A}$ be a Grothendieck category with a generator $G$ and put $R = \ehom_\mathcal{A}(G)$. Then the functor $H = \hom_\mathcal{A}(G,-)\colon \mathcal{A} \to \Mod R$ is fully faithful and has an exact left adjoint $T$,

\vspace{0.25cm}
\begin{adjustbox}{max totalsize={1\textwidth}{.9\textheight},center}
\begin{tikzcd}
\Mod R \ar[rr, bend left, start anchor=north east, end anchor=north west, "T"] & \perp & \mathcal{A}\;. \ar[ll, bend left, start anchor=south west, end anchor=south east, hook, "H"]
\end{tikzcd}
\end{adjustbox}
\end{prop}

In particular, the adjunction counit $\varepsilon\colon TH\to 1_\mathcal{A}$ is a natural equivalence and $H$ identifies $\mathcal{A}$ with an extension closed reflective subcategory of $\Mod R$, where the adjunction unit $\eta\colon 1_{\Mod R} \to HT$ provides the reflections.

Clearly, $T$ preserves all colimits and finite limits in $\mathcal{A}$, but it need not preserve infinite products. However, the following easy observation shows that it does preserve at least certain products.

\begin{lem}\label{lem:GabrielPopescuProducts}
Suppose that we are in the situation of Proposition~\ref{prop:GabrielPopescu}. If $M_i,$ $i\in I,$ is a family of $R$-modules in the essential image of $H$, then the canonical morphism $T(\prod_i M_i) \to \prod_i T(M_i)$ is an isomorphism.
\end{lem}

\begin{proof}
Let $\im H\subseteq \Mod R$ denote the essential image of $H$. The lemma follows from the fact that $T$ is an inverse of the category equivalence $H\colon \mathcal{A} \to \im H$ and that products in $\im H$ are also products in $\Mod R$.
\end{proof}

Now we can characterize pure-injectivity as follows.

\begin{prop}\label{prop:PureInjChar}
Let $\mathcal{A}$ be a Grothendieck category and $C \in \mathcal{A}.$ Then the following are equivalent:

\begin{enumerate}[(1)]
\item\label{PInj1}{$C$ is pure-injective in $\mathcal{A}$.}
\item\label{PInj2}{There is a generator $G\in\mathcal{A}$ such that $\hom_\mathcal{A}(G, C)$ is a pure-injective right $\ehom_\mathcal{A}(G)$-module.}
\item\label{PInj3}{$\hom_\mathcal{A}(G, C)$ is a pure-injective $\ehom_\mathcal{A}(G)$-module for each generator $G\in\mathcal{A}$.}
\end{enumerate}
\end{prop}

\begin{proof}
Let $G\in\mathcal{A}$ be a generator of $\mathcal{A}$, $R=\ehom_\mathcal{A}(G)$, and consider the fully faithful functor $H=\hom_\mathcal{A}(G,-)\colon \mathcal{A} \to \Mod{R}$. If $I$ is any set and we denote $M=H(C)$, we have a commutative square

\vspace{0.25cm}
\begin{adjustbox}{max totalsize={1\textwidth}{.9\textheight},center}
	\begin{tikzcd}
		M^{\oplus I} \ar[r] \ar[d, "\eta'"]& M^{\times I}\ar[d, "\simeq"] \\
		H(C^{\oplus I}) \ar[r]& H(C^{\times I})\;.
	\end{tikzcd}
\end{adjustbox}
\vspace{0.15cm}

\noindent
The horizontal morphisms are the canonical ones, while the vertical ones are the compositions of the unit of the adjunction with the isomorphisms $HT(M^{\oplus I}) \to H((TM)^{\oplus I}) \cong H(C^{\oplus I})$ and $HT(M^{\times I}) \to H((TM)^{\times I}) \cong H(C^{\times I})$, respectively. In particular, $\eta'$ is an $\im H$-reflection and $\hom_\mathcal{A}(\eta',M)$ is an isomorphism.

Now $\Sigma \in \hom_\mathcal{A}(C^{\oplus I}, C)$ factorizes through $C^{\oplus I} \hookrightarrow C^{\times I}$ if and only if $H(\Sigma)$ factorizes through the lower horizontal map of the square if and only if the summing map $\Sigma'\colon M^{\oplus I}\to M$ factorizes through the upper horizontal map of the square. 

The latter condition characterizes pure-injectivity of $M$ as an $R$-module by~\cite[Theorem 7.1~(vi)]{JensenLenzing} and the argument clearly does not depend on the choice of the generator $G$.
\end{proof}

Our next concern is to prove that cotilting classes are closed under taking direct limits. We first introduce a special kind of direct limits.

\begin{deff}\label{def:ReducedProd}
Let $I$ be a set, $(F_i\mid i\in I)$ be a family of objects of a Grothendieck category $\mathcal{A}$ and $\kappa$ be an infinite cardinal number. A \emph{$\kappa$-bounded product} $\prod_{i \in I}^{<\kappa} F_i$ is defined as the direct limit $\varinjlim_{J\subseteq I} \prod_{j\in J} F_j,$ where $J$ runs over all subsets of $I$ of cardinality $<\kappa$ and the maps in the direct system are the canonical split embeddings $\prod_{j\in J} F_j \hookrightarrow \prod_{j\in J'} F_j$ for each $J\subseteq J'\subseteq I$.

Clearly, the bounded product $\prod_{i \in I}^{<\kappa} F_i$ canonically embeds into the usual product $\prod_{i\in I} F_i$. The factor $\prod_{i\in I} F_i/\prod_{i\in I}^{<\kappa} F_i$ is called the \emph{$\kappa$-reduced product}.

When all $F_\alpha$ are equal to a single object $F$, we will speak of \emph{$\kappa$-bounded} and \emph{$\kappa$-reduced powers}
 of $F$, respectively. If $\kappa$ is clear from the context, we will denote these by $F^{\boxtimes I}$ and $F^{\times I}/F^{\boxtimes I}$, respectively.
\end{deff}

If $\kappa = \aleph_0$, the bounded product $\prod_{i\in I}^{<\aleph_0} F_i$ coincides with the usual direct sum $\bigoplus_{i\in I} F_i$. If $\mathcal{A}$ is a module category, the bounded product $\prod_{i \in I}^{<\kappa} F_i$ can be described as the submodule of $\prod_{i\in I} F_i$ formed by the elements with $<\kappa$ non-zero components.

We will employ the following relation between reduced products and direct limits which is essentially a special case of~\cite[Theorem 3.3.2]{Prest}.

\begin{prop}\label{prop:ReducedProdToDirectLim}
Let $\mathcal{A}$ be a Grothendieck category, $\kappa$ be an infinite regular cardinal, $(F_\alpha, f_{\alpha\beta}\colon F_\alpha\to F_\beta \mid \alpha<\beta<\kappa)$ be a direct system in $\mathcal{A}$ indexed by $\kappa$ and put $F=\prod_{\alpha<\kappa} F_\alpha$. Then there is an embedding into the $\kappa$-reduced product $F$,

\vspace{0.25cm}
\begin{adjustbox}{max totalsize={1\textwidth}{.9\textheight},center}
	\begin{tikzcd}
		\varinjlim\limits_{\alpha<\kappa} F_\alpha \ar[r, hook]& \prod\limits_{\alpha<\kappa} F^{\times \kappa}/F^{\boxtimes \kappa}.
	\end{tikzcd}
\end{adjustbox}
\end{prop}

\begin{proof}
We can transfer the problem to a module category thanks to Proposition~\ref{prop:GabrielPopescu}. Let $M_\alpha = H(F_\alpha)$ for each $\alpha<\kappa$ and $M = H(F) \simeq \prod_{\alpha<\kappa} M_\alpha$. Then we have embeddings (even pure embeddings)

\vspace{0.25cm}
\begin{adjustbox}{max totalsize={1\textwidth}{.9\textheight},center}
	\begin{tikzcd}
		\varinjlim\limits_{\alpha<\kappa} M_\alpha \ar[r, hook]& \prod\limits_{\alpha<\kappa} M_\alpha/{\prod\limits_{\alpha<\kappa}}\!\!^{<\kappa} M_\alpha  \ar[r, hook]&
		M^{\times \kappa}/M^{\boxtimes \kappa}
	\end{tikzcd}
\end{adjustbox}
\vspace{0.15cm}

\noindent in $\Mod{R}$, where the first one is obtained from (the proof of) \cite[Theorem 3.3.2]{Prest} and the second one is induced by the product map $\prod_{\alpha<\kappa}\iota_\alpha\colon \prod_{\alpha<\kappa} M_\alpha \hookrightarrow M^{\times \kappa}$ of the split embeddings $\iota_\alpha\colon M_\alpha \hookrightarrow M$.

Since $T\colon \Mod{R} \to \mathcal{A}$ is exact, preserves all colimits and, by Lemma~\ref{lem:GabrielPopescuProducts}, also products of copies the module $M$, we obtain the embedding from the statement simply by an application of $T$.
\end{proof}

\begin{cor}\label{cor:ReducedProdToDirectLim}
Given a Grothendieck category $\mathcal{A}$ and a class of objects $\mathcal{F}\subseteq\mathcal{A}$ closed under products and subobjects, $\mathcal{F}$ is closed under taking direct limits if and only if the following holds:

Given an infinite regular cardinal $\kappa$ and an object $F \in \mathcal{F}$, the $\kappa$-reduced power 
$M^{\times \kappa}/M^{\boxtimes \kappa}$ belongs to $\mathcal{F}$.
\end{cor}

\begin{proof}
The `if' part is clear since we can express the $\kappa$-reduced power of $F$ as $\varinjlim_{\alpha<\kappa} F^{\times (\kappa\setminus\alpha)}$.

Conversely, if $\mathcal{F}$ is closed under subobjects, products and the reduced powers as above, Proposition~\ref{prop:ReducedProdToDirectLim} implies that it is also closed under direct limits of well ordered chains indexed by infinite regular cardinals. Then $\mathcal{F}$ is in fact closed under direct limits of all well-ordered chains, for any such chain indexed by a limit ordinal $\lambda$ has a cofinal subchain indexed by $\kappa = \operatorname{cf}(\lambda)$, the cofinality of $\lambda$, and $\kappa$ is known to be an infinite regular cardinal. It follows from (the proof of) \cite[Corollary 1.7]{AdamekRosicky} that $\mathcal{F}$ is closed under all direct limits.
\end{proof}

Since any cotilting class $\mathcal{F}=\cogen({C})$ is a torsion-free class, we have reduced the original problem to proving that cotilting classes are closed under certain reduced powers. To that end, we use the following result which comes from \cite{BazzoniPureInj} for the case $\kappa=\aleph_0$ and from \cite{StPureInj} for a general $\kappa$.

\begin{prop} \label{prop:CountingSequence}
Let $\mathcal{A}$ be a Grothendieck category, $F\in \mathcal{A}$ be an object and $\kappa$ be an infinite regular cardinal. Then there exist arbitrarily large cardinal numbers $\lambda$ such that there is an exact sequence

\vspace{0.25cm}
\begin{adjustbox}{max totalsize={1\textwidth}{.9\textheight},center}
	\begin{tikzcd}
		0 \ar[r]& F^{\boxtimes\lambda} \ar[r]& E \ar[r]& P^{\oplus 2^\lambda} \ar[r]& 0\;,
	\end{tikzcd}
\end{adjustbox}
\vspace{0.15cm}

\noindent where $F^{\boxtimes\lambda}$ is the $\kappa$-bounded product of $\lambda$ copies of $F$, $F^{\boxtimes \lambda} \subseteq E \subseteq F^{\times \lambda}$, and $P = F^{\times\kappa}/F^{\boxtimes\kappa}$ is the $\kappa$-reduced product of $\kappa$ copies of $F$.
\end{prop}

\begin{proof}
We again use Proposition~\ref{prop:GabrielPopescu} and put $M=H(F).$ A construction of such an exact sequence for $M\in\Mod{R}$ in place of $F$ is given in the proof of \cite[Lemma 7]{StPureInj}, and then we apply $T$ to transfer this exact sequence back to $\mathcal{A}$.
\end{proof}

Now we can prove the main result of the section.

\begin{thm}\label{thm:CotiltingPureInjective}
Let $\mathcal{A}$ be a Grothendieck category, $C\in\mathcal{A}$ be a cotilting object and $\mathcal{F} = \cogen({C}) = {^\perp C}$ the associated cotilting class (Definition~\ref{def:cotilting}). Then $C$ is pure-injective and $\mathcal{F}$ is closed under direct limits in $\mathcal{A}$.
\end{thm}

\begin{proof}
Once we prove that $\mathcal{F}$ is closed under direct limits, the pure-injectivity of $C$ will follow. Indeed, in that case
\[ C^{\times I}/C^{\oplus I} = \varinjlim_{F\subseteq I} C^{\times (I\setminus F)} \in \mathcal{F}, \]
where $F$ runs over all finite subsets of $I$, and hence every morphism $C^{\oplus I} \to C$ extends to a morphism $C^{\times I} \to C$.

In view of Corollary~\ref{cor:ReducedProdToDirectLim}, we only need to prove that given $F\in\mathcal{F}$, we have that $P = F^{\times\kappa}/F^{\boxtimes\kappa}$, the $\kappa$-reduced product of $\kappa$ copies of $F$, lies in $\mathcal{F}$. To this end, we will use a variant of the argument from~\cite[Proposition 2.5]{BazzoniPureInj} and \cite[Lemma 7]{StPureInj} (the method can be traced back to \cite{Hunter}).

Let $\lambda\ge\kappa$ be an infinite cardinal number for which there exists an exact sequence from Proposition~\ref{prop:CountingSequence} and such that $\lambda \ge \card{\hom_\mathcal{A}(F^{\times \mu}, C)}$ for each $\mu<\kappa$. Since each morphism $F^{\boxtimes \lambda} \to C$ is, by the universal property of the defining colimit, uniquely determined by its compositions with the embeddings $F^{\times J} \hookrightarrow F^{\boxtimes \lambda}$, where $J\subseteq \lambda$ and $\card{J}<\kappa$, and there are at most $\lambda^\kappa \le (2^\lambda)^\kappa = 2^{\lambda\times\kappa} = 2^\lambda$ such embeddings, we have
\[
\card{\hom_\mathcal{A}(F^{\boxtimes \lambda}, C)} \le
2^\lambda \times \lambda = 2^\lambda.
\]
If we on the other hand apply $\hom_\mathcal{A}(-, C)$ to the exact sequence from Proposition~\ref{prop:CountingSequence} and use Remark~\ref{rem:ExtCoproductsOk}, we obtain an exact sequence

\vspace{0.25cm}
\begin{adjustbox}{max totalsize={1\textwidth}{.9\textheight},center}
	\begin{tikzcd}
		\hom_\mathcal{A}(F^{\boxtimes \lambda}, C) \ar[r]&
		\ext^1_\mathcal{A}(P, C)^{\times 2^\lambda} \ar[r]&
		\ext^1_\mathcal{A}(E, C)\;.
	\end{tikzcd}
\end{adjustbox}
\vspace{0.15cm}

\noindent Since $E \subseteq F^{\times \lambda} \in \mathcal{F}$, we have $\ext^1_\mathcal{A}(E, C) = 0$, so the group $\ext^1_\mathcal{A}(P, C)^{\times 2^\lambda}$ is an epimorphic image of $\hom_\mathcal{A}(F^{\boxtimes \lambda}, C)$. Now if $\ext^1_\mathcal{A}(P, C) \ne 0$, we would have
\[ \card{\ext^1_\mathcal{A}(P, C)^{\times 2^\lambda}} \ge 2^{2^\lambda}, \]
which is more than the cardinality of $\hom_\mathcal{A}(F^{\boxtimes \lambda}, C)$. Thus $\ext^1_\mathcal{A}(P, C) = 0$ and $P = F^{\times\kappa}/F^{\boxtimes\kappa} \in \mathcal{F}$, as required.
\end{proof}

Our application of the latter theorem is the following description of which torsion-free classes are associated with cotilting objects. This (together with Theorem~\ref{thm:maintheorem} in the next section) extends \cite[Proposition 5.7]{ParraSaorin} to all Grothendieck categories.

\begin{thm}\label{thm:CharTiltingClasses}
Let $\mathcal{A}$ be a Grothendieck category. Then the following are equivalent for a full subcategory $\mathcal{F}\subseteq\mathcal{A}$:
	\begin{enumerate}[(1)]
		\item\label{CotiltChar1}{$\mathcal{F}$ is a cotilting class associated with a cotilting object.}
		\item\label{CotiltChar2}{$\mathcal{F}$ is a torsion-free class in $\mathcal{G}$ which contains a generator and is closed under direct limits.}
	\end{enumerate}
\end{thm}

The implication \hyperref[CotiltChar1]{(\ref*{CotiltChar1})}$\Rightarrow$\hyperref[CotiltChar2]{(\ref*{CotiltChar2})} is an immediate consequence of Theorem~\ref{thm:CotiltingPureInjective}.
Our main technical tool to prove the other implication, and in particular to construct corresponding cotilting objects, is a result on covering classes in Grothendieck categories (this implication is in fact already included in \cite[Proposition 5.7]{ParraSaorin}, but we give a more direct construction of the cotilting objects which we later use in Section~\ref{sec:classif}).

Given a class of objects $\mathcal{F}$ in a category $\mathcal{A}$ and $X\in\mathcal{A}$, an \emph{$\mathcal{F}$-precover} of $X$ is a morphism $f\colon F \to X$ such that $\hom_\mathcal{A}(F',f)$ is surjective for each $F'\in\mathcal{F}$. An $\mathcal{F}$-precover $f\colon F \to X$ is an \emph{$\mathcal{F}$-cover} provided that each endomorphism $g\colon F \to F$ such that $fg=f$ is necessarily an automorphism. The class $\mathcal{F}$ is called \emph{(pre)covering} if each $X\in\mathcal{A}$ has an $\mathcal{F}$-(pre)cover.
Given a full subcategory $\mathcal{X}$ of a Grothendieck category $\mathcal{A}$, we will denote by $\Limcl\mathcal{X}$ the class of all direct limits of objects in $\mathcal{X}$.

\begin{prop}[{\cite[Theorem~3.2]{ElBashir}}]\label{prop:ElBashir}
Let $\mathcal{A}$ be a Grothendieck category and $\mathcal{F} \subseteq \mathcal{A}$ be a class of objects closed under direct sums and direct limits. Suppose that there is a set $\mathcal{S} \subseteq \mathcal{F}$ such that $\mathcal{F}=\Limcl\mathcal{S}$. Then $\mathcal{F}$ is a covering class.
\end{prop}

If $\mathcal{F}$ is extension closed and generating, we can use the following lemma which is originally due to Wakamatsu~\cite{Wakamatsu}.

\begin{lem}[{\cite[Lemma 5.13]{G-T}}]\label{lem:Wakamatsu}
Let $\mathcal{A}$ be an abelian category, $\mathcal{F}\subseteq\mathcal{A}$ be a generating class closed under extensions and let $f\colon F \to A$ be an $\mathcal{F}$-cover of an object $A \in \mathcal{A}$. Then $f$ is an epimorphism and $\ext^1_\mathcal{A}(\mathcal{F}, \ker f) = 0$.
\end{lem}

\begin{proof}[Proof of Theorem~\ref{thm:CharTiltingClasses}]
In view of Theorem~\ref{thm:CotiltingPureInjective}, it remains to prove that given a torsion-free class $\mathcal{F} \subseteq \mathcal{A}$ which contains a generator and is closed under direct limits, it is associated to a cotilting object $C \in \mathcal{A}$.
Note that under these assumptions, there is a set $\mathcal{S} \subseteq \mathcal{F}$ such that $\mathcal{F}=\Limcl\mathcal{S}$.
Indeed, fix a generator $G \in \mathcal{A}$ and let $\mathcal{S} \subseteq \mathcal{F}$ be the set of all quotients of finite coproducts of $G$ that are in $\mathcal{F}$. Then for $F \in \mathcal{F}$ and an epimorphism $G^{\oplus I} \rightarrow  F$, lettting $F_{I_0}:=\Im (G^{\oplus I_0} \hookrightarrow G^{\oplus I} \rightarrow F)$ for each finite subset $I_0 \subseteq I$ yields a direct system of objects from $\mathcal{S}$ with $\varinjlim_{I_0}F_{I_0}=F$. Thus, $\mathcal{F} \subseteq \Limcl \mathcal{S}$, and the converse inclusion follows from $\mathcal{F}$ being closed under direct limits.

The construction of the cotilting module is very similar to the one in~\cite[Theorem 15.22]{G-T}. Let $W$ be an injective cogenerator of $\mathcal{A}$ and consider an exact sequence

\vspace{0.25cm}
\begin{adjustbox}{max totalsize={1\textwidth}{.9\textheight},center}
	\begin{tikzcd} 
\varepsilon\colon	&	0 \ar[r] & C_1 \ar[r] & C_0 \ar[r, "\pi"] & W \ar[r] & 0\;, &
	\end{tikzcd}
\end{adjustbox}	
\vspace{0.15cm}

\noindent
where $\pi$ is an $\mathcal{F}$-cover. Then $C_1 \in \mathcal{F}^\perp$ by Lemma~\ref{lem:Wakamatsu} and $C_0 \in \mathcal{F}^\perp$ as $\mathcal{F}^\perp$ is closed under extensions. We put $C = C_0 \oplus C_1$ and observe (using Corollary~\ref{cor:BunoProd}) that
\[
\Prodcl({C}) \subseteq \mathcal{F}\cap\mathcal{F}^\perp
\quad\text{ and }\quad
\cogen({C}) \subseteq \mathcal{F} \subseteq {^\perp C}\;. \]

Suppose further that $A \in {^\perp C}$. We use the same method as in the proof of Theorem~\ref{thm:CotiltingChar} to show that $A \in \cogen({C})$. Indeed, $\cogen({C})$ is a torsion-free class in $\mathcal{A}$ by Lemma~\ref{lem:cotilting-TF}, so we have an exact sequence

\vspace{0.25cm}
\begin{adjustbox}{max totalsize={1\textwidth}{.9\textheight},center}
	\begin{tikzcd} 
		0 \ar[r] &  T \ar[r] & A \ar[r] & F  \ar[r] & 0\;
	\end{tikzcd}
\end{adjustbox}	
\vspace{0.15cm}

\noindent with ${F} \in \cogen({C})$ and ${T} \in \ker \hom_{\mathcal{A}}(-, {C}).$
Since also $\injdim C\le 1$ by Proposition~\ref{prop:cotilting-injdim1}, we also have $T \in {^\perp C}$ and application of $\hom_{\mathcal{A}}({T}, -)$ to the sequence $\varepsilon$ yields an exact sequence

\vspace{0.25cm}
\begin{adjustbox}{max totalsize={1\textwidth}{.9\textheight},center}
	\begin{tikzcd}
		0 = \hom_{\mathcal{A}}({T}, {C}_0) \ar[r] & \hom_{\mathcal{A}}({T}, {W}) \ar[r] &  \ext^1_{\mathcal{A}}({T}, {C}_1) = 0 \;.
	\end{tikzcd}
\end{adjustbox}
\vspace{0.15cm}

It follows that ${T}=0$ and $A \simeq F \in \cogen{(C)}$. Thus, $\mathcal{F} = \cogen({C}) = {^\perp C}$ and $C$ is a cotilting object.
\end{proof}

In the locally Noetherian case, we obtain a classification of cotilting torsion-free classes in a locally Noetherian category via torsion pairs of Noetherian objects, which recovers \cite[Corollary 5.13]{ParraSaorin}.
An almost identical result appeared in \cite[Theorem 1.13]{BuanKrause-cotilting}, but with a somewhat different definition of a cotilting object (we were not able to recover the result without replacing (C3) from~\cite{BuanKrause-cotilting} by the slightly stronger condition in Theorem~\hyperref[C3]{\ref*{thm:CotiltingChar}~(\ref*{Char2})~(C3)}, but on the other hand we proved in Theorem~\ref{thm:CotiltingPureInjective} that condition (C4) from~\cite{BuanKrause-cotilting} was superfluous).
A special case for module categories of (one-sided) Noetherian rings was also obtained in \cite[Proposition 2.6]{St1}.
Geometric consequences for categories of quasi-coherent sheaves will be discussed later in Sections~\ref{sec:torsion} and~\ref{sec:classif}.

\begin{thm}\label{thm:ClassificationViaTP}
Let $\mathcal{A}$ be a locally Noetherian Grothendieck category and $\mathcal{A}_0$ be the full subcategory of Noetherian objects. Then torsion-free classes $\mathcal{F}$ in $\mathcal{A}$ associated to a cotilting object bijectively correspond to the torsion pairs $(\mathcal{T}_0, \mathcal{F}_0)$ in $\mathcal{A}_0$ for which $\mathcal{F}_0$ is a generating class (i.e.\ each object of $\mathcal{A}_0$ is a quotient of an object of $\mathcal{F}_0$). The correspondence is given by
\[ \mathcal{F}\mapsto\mathcal{F}\cap\mathcal{A}_0  \quad\text{ and }\quad (\mathcal{T}_0, \mathcal{F}_0) \mapsto \Limcl\mathcal{F}_0. \]
\end{thm}

In order to prove the result, we first make precise the relation of torsion pairs in $\mathcal{A}_0$ to torsion pairs in $\mathcal{A}$.

\begin{lem}\label{lem:TorsionPairsQcohvsCoh}
	Let $\mathcal{A}$ be a locally Noetherian category and let $\mathcal{A}_0$ be the full subcategory of Noetherian objects.
	\begin{enumerate}[(1)]
		\item\label{NLink1}{If $(\mathcal{T}, \mathcal{F})$ is a torsion pair in $\mathcal{A}$, then $(\mathcal{T}\cap \mathcal{A}_0, \mathcal{F}\cap \mathcal{A}_0)$ is a torsion pair in $\mathcal{A}_0$.}
		\item\label{NLink2}{If $(\mathcal{T}_0, \mathcal{F}_0)$ is a torsion pair in $\mathcal{A}_0$, then $(\Limcl\mathcal{T}_0, \Limcl\mathcal{F}_0)$ is a torsion pair in $\mathcal{A}$.}
	\end{enumerate}
	Moreover, these assignments yield a bijective correspondence between the class of all torsion pairs in $\mathcal{A}_0$, and the class of those torsion pairs $(\mathcal{T}, \mathcal{F})$ in $\mathcal{A}$ for which $\mathcal{F}$ is closed under direct limits. Furthermore, $\mathcal{F}_0$ is generating (in $\mathcal{A}_0$ or in $\mathcal{A}$) if and only if the corresponding class $\mathcal{F}$ contains a generator for $\mathcal{A}$.
\end{lem}

\begin{proof}
	Part~\hyperref[NLink1]{(\ref*{NLink1})} is clear since the torsion and the torsion-free part of a Noetherian object are both Noetherian. A proof of part~\hyperref[NLink2]{(\ref*{NLink2})} can be found in~\cite[\S4.4]{CBLocFinPres} (the distinction between direct limits here and filtered colimits used in~\cite{CBLocFinPres} is inessential thanks to~\cite[Theorem 1.5]{AdamekRosicky}). 
		
	In order to prove that we have the bijective correspondence, observe first that $\mathcal{T}_0 = \mathcal{A}_0\cap \Limcl\mathcal{T}_0$ and $\mathcal{F}_0 = \mathcal{A}_0\cap \Limcl\mathcal{F}_0$ (see for instance ~\cite[\S4.1]{CBLocFinPres}). Note that the class $\Limcl\mathcal{F}_0$ can be described as
$$\Limcl\mathcal{F}_0=\ker \hom_{\mathcal{A}}(\mathcal{T}_0, -)=\bigcap_{T \in \mathcal{T}_0}\ker \hom_{\mathcal{A}}(T, -).$$
Since $\mathcal{T}_0$ consists of Noetherian, hence finitely presented objects only, it follows that $\Limcl\mathcal{F}_0$ is always closed under taking direct limits. 

 If, on the other hand, $(\mathcal{T}, \mathcal{F})$ is a torsion pair in $\mathcal{A}$ with $\mathcal{F}$ closed under direct limits and if $(\mathcal{T}_0, \mathcal{F}_0)$ is its restriction to $\mathcal{A}_0$, then clearly $\Limcl\mathcal{T}_0 \subseteq \mathcal{T}$ and $\Limcl\mathcal{F}_0 \subseteq \mathcal{F}$. It follows from part~\hyperref[NLink2]{(\ref*{NLink2})} and  Lemma~\ref{lem:TorsionPairsInclusion} that $\Limcl\mathcal{T}_0 = \mathcal{T}$ and $\Limcl\mathcal{F}_0 = \mathcal{F}$.
 
Finally, note that if $G \in \mathcal{F}$ is a generator, then the set of all noetherian subobjects of $G$ is a generating subset of $\mathcal{F}_0=\mathcal{F}\cap \mathcal{A}_0$, and conversely, if $\mathcal{F}_0$ is generating, $\mathcal{A}$ having a \emph{set} of noetherian generators implies that one can choose a set $\mathcal{S} \subseteq \mathcal{F}_0$ that is generating; then $\bigoplus_{X \in \mathcal{S}} X$ is a generator in $\mathcal{F}$.
\end{proof}

\begin{proof}[Proof of Theorem~\ref{thm:ClassificationViaTP}]
Lemma~\ref{lem:TorsionPairsQcohvsCoh} implies that the assignment $\mathcal{F}\mapsto\mathcal{F}\cap\mathcal{A}_0$ from Theorem~\ref{thm:ClassificationViaTP} is injective since $\mathcal{F}$ can be reconstructed as $\Limcl(\mathcal{F} \cap \mathcal{A}_0)$.
The surjectivity follows from Theorem~\ref{thm:CharTiltingClasses} together with  Lemma~\ref{lem:TorsionPairsQcohvsCoh} since $\Limcl\mathcal{F}_0$ is always associated with a cotilting object (and $\mathcal{F}_0=\mathcal{A}_{0} \cap \Limcl\mathcal{F}_0$).
\end{proof}


\section{Derived equivalences}
\label{sec:derived-equiv}

Now we turn to derived equivalences. Most of the material in this section is rather standard nowadays, with one notable exception: we prove that a cotilting object in a Grothendieck category induces a derived equivalence to another Grothendieck category. This uses pure-injectivity of cotilting objects in a crucial way.

We first summarize the essentials of tilting theory for torsion pairs. Given a torsion pair $(\mathcal{T}, \mathcal{F})$ in an abelian category $\mathcal{A}$, there is a procedure, worked out in \cite{HRS,BondalVanDenBergh} and also treated in~\cite{ColpiHeart,Noohi,StovicekKernerTrlifaj}, to construct another abelian category $\mathcal{H}$.

\begin{prop}[{\cite[\S1.3]{HRS} and \cite[\S\S5.3--5.4]{BondalVanDenBergh}}] \label{prop:HRS-tilt}
Let $\mathcal{A}$ be an abelian category, $(\mathcal{T}, \mathcal{F})$ a torsion pair in $\mathcal{A}$, $\Der({\mathcal{A}})$ the (unbounded) derived category of  $\mathcal{A}$ with the suspension functor $\Sigma$, and put
\[ \mathcal{H} = \{ X\in\Der(\mathcal{A}) \mid H^0(X)\in\mathcal{F}, H^{1}(X)\in\mathcal{T} \text{ and } H^i(X)=0 \text{ for } i \ne 0, 1 \}\;. \]
Then $\mathcal{H}$ is itself an abelian category with a torsion pair $(\mathcal{F}, \Sigma^{-1}\mathcal{T})$. If, moreover, $\mathcal{F}$ is a generating class in $\mathcal{A}$ (i.e.\ each object of $\mathcal{A}$ is a quotient of an object of $\mathcal{F}$), then $\mathcal{F}$ is a cogenerating class in $\mathcal{H}$ and the embedding $\mathcal{H}\subseteq \Der(\mathcal{A})$ extends to an exact equivalence

\vspace{0.25cm}
\begin{adjustbox}{max totalsize={1\textwidth}{.9\textheight},center}
	\begin{tikzcd}
		\Der(\mathcal{H}) \ar[r, "\simeq"]&
		\Der(\mathcal{A})\;.
	\end{tikzcd}
\end{adjustbox}
\end{prop}

\begin{deff}\label{def:HRS-tilt}
The category $\mathcal{H}$ from the proposition is called \emph{tilted from $\mathcal{A}$} in the sense of Happel, Reiten and Smal\o{} (or \emph{HRS-tilted} for short).
\end{deff}

We also record a few standard facts which are very useful for computations in HRS-tilted categories.

\begin{lem}\label{lem:HRS-exts}
Let $\mathcal{A}$ be an abelian category, $(\mathcal{T}, \mathcal{F})$ a torsion pair such that $\mathcal{F}$ is generating, and $\mathcal{H}$ the HRS-tilted category. Given $T, T'\in \mathcal{T},$ $F, F'\in \mathcal{F}$ and $n\ge 0,$ we have isomorphisms
\begin{enumerate}[(1)]
\item\label{HRSExtTT}{$\ext^n_\mathcal{H}(\Sigma^{-1}T, \Sigma^{-1}T') \simeq \ext^n_\mathcal{A}(T, T'),$}
\item\label{HRSExtFF}{$\ext^n_\mathcal{H}(F, F') \simeq \ext^n_\mathcal{A}(F, F'),$}
\item\label{HRSExtTF}{$\ext^n_\mathcal{H}(\Sigma^{-1}T, F') \simeq \ext^{n+1}_\mathcal{A}(T, F')$ and}
\item\label{HRSExtFT}{$\ext^{n+1}_\mathcal{H}(F, \Sigma^{-1}T') \simeq \ext^n_\mathcal{A}(F, T').$}
\end{enumerate}
Moreover, a sequence of the form $0 \to F \to F'' \to F' \to 0$ is exact in $\mathcal{A}$ if and only if it is exact in $\mathcal{H}$.
\end{lem}

\begin{proof}
Using the derived equivalence $\Der(\mathcal{H}) \simeq \Der(\mathcal{A})$ from Proposition~\ref{prop:HRS-tilt}, we have isomorphisms
\begin{multline*}
\ext^n_\mathcal{H}(\Sigma^{-1}T, \Sigma^{-1}T') \simeq
\hom_{\Der(\mathcal{H})}(\Sigma^{-1}T, \Sigma^{n-1}T') \simeq \\ \simeq
\hom_{\Der(\mathcal{A})}(T, \Sigma^{n}T') \simeq
\ext^n_\mathcal{A}(T, T')\;.
\end{multline*}
The other isomorphisms are proved in a similar way.

If $\varepsilon\colon 0 \to F \overset{i}\to F'' \overset{p}\to F' \to 0$ is short exact either in $\mathcal{A}$ or in $\mathcal{H}$, we have $F'' \in \mathcal{F}$ and a triangle

\vspace{0.25cm}
\begin{adjustbox}{max totalsize={1\textwidth}{.9\textheight},center}
	\begin{tikzcd}
		F \ar[r, "i"]& F'' \ar[r, "p"]& F' \ar[r, "{[\varepsilon]}"]& \Sigma F
	\end{tikzcd}
\end{adjustbox}
\vspace{0.15cm}

\noindent in $\Der(\mathcal{A})$ or $\Der(\mathcal{H})$, respectively. Thanks to the derived equivalence, we have the same triangle in the other derived category and, hence, the same short exact sequence in the other abelian category.
\end{proof}

Now we can make precise what role cotilting objects play in the HRS-tilted category (see \cite[proof of Proposition 5.7]{ParraSaorin} for analogous considerations and \cite[Proposition 3.8]{ColpiHeart} for a version of the result if $\mathcal{A}$ is a module category).

\begin{prop}\label{prop:CotiltingVSInjective}
Let $\mathcal{A}$ be a Grothendieck category, $(\mathcal{T}, \mathcal{F})$ a torsion pair with $\mathcal{F}$ generating, and $\mathcal{H}$ the HRS-tilted category. Then the following hold for an object $C \in \mathcal{H}$:
\begin{enumerate}
\item\label{HRSInjObj}{$C$ is injective in $\mathcal{H}$ if and only if $C \in \mathcal{F}$ and $\ext^1_\mathcal{A}(\mathcal{F}, C) = 0$.}
\item\label{HRSInjCog}{$C$ is an injective cogenerator of $\mathcal{H}$ if and only if $C$ is a cotilting object in $\mathcal{A}$ with $\mathcal{F}=\cogen({C})$.}
\end{enumerate}
\end{prop}

\begin{proof}
\hyperref[HRSInjObj]{(\ref*{HRSInjObj})}
If $C$ is injective in $\mathcal{H}$, it is a summand of an object in $\mathcal{F}$ since $\mathcal{F}$ is cogenerating in $\mathcal{H}$. In particular, $C\in\mathcal{F}$. Moreover, $C\in{\mathcal{F}^\perp}$ in $\mathcal{A}$ by Lemma~\hyperref[HRSExtFF]{\ref*{lem:HRS-exts}~(\ref*{HRSExtFF})}.

Conversely, if $C\in \mathcal{F}\cap\mathcal{F}^\perp$ in $\mathcal{A}$, then the injective dimension of $C$ in $\mathcal{A}$ is at most one by Proposition~\ref{prop:cotilting-injdim1}. Moreover,
\[ \ext^1_\mathcal{H}(\mathcal{F}, C) = 0 \quad\text{ and }\quad \ext^1_\mathcal{H}(\Sigma^{-1}\mathcal{T}, C) \simeq \ext^2_\mathcal{A}(\mathcal{T}, C) = 0 \]
by Lemma~\hyperref[HRSExtFF]{\ref*{lem:HRS-exts}~(\ref*{HRSExtFF})} and~\hyperref[HRSExtTF]{(\ref*{HRSExtTF})}. It follows that $C$ is injective in $\mathcal{H}$ since each $X\in\mathcal{H}$ is an extension in $\mathcal{H}$ of an object of $\mathcal{F}$ by an object of $\Sigma^{-1}\mathcal{T}$. 

\hyperref[HRSInjCog]{(\ref*{HRSInjCog})}
Suppose first that $C\in\mathcal{A}$ is a cotilting object associated with $\mathcal{F}$. Since each product of copies of $C$ in $\mathcal{A}$ coincides with the corresponding product in $\Der(\mathcal{A})\simeq \Der(\mathcal{H})$ by Corollary~\ref{cor:ExactProdOfC}, we have $\Prodcl({C})\subseteq \mathcal{H}$. In particular, arbitrary products of copies of $C$ exist in $\mathcal{H}$ and agree with the ones in $\mathcal{A}$. Moreover, $\Prodcl({C})$ consists of injective objects by part~\hyperref[HRSInjObj]{(\ref*{HRSInjObj})} and each object of $\mathcal{F}$ is a subobject in $\mathcal{H}$ of a product of copies of $C$ by Lemmas~\ref{lem:cogen=copres} and~\ref{lem:HRS-exts}. Since $\mathcal{F}$ is itself a cogenerating class in $\mathcal{H}$, $C$ is an injective cogenerator of $\mathcal{H}$.

Conversely, let $C\in\mathcal{H}$ be an injective cogenerator in $\mathcal{H}$. We know from part~\hyperref[HRSInjObj]{(\ref*{HRSInjObj})} that $C\in\mathcal{F}\cap\mathcal{F}^\perp$ in $\mathcal{A}$ and the injective dimension of $C$ in $\mathcal{A}$ is at most one. Since $\mathcal{F}$ is a torsion-free class in $\mathcal{A}$, we have $\cogen({C}) \subseteq \mathcal{F} \subseteq {^\perp{C}}$ in $\mathcal{A}$.

We first show that $\mathcal{F} = {^\perp{C}}$. Suppose now that $A \in {^\perp{C}}$ in $\mathcal{A}$ and consider an exact sequence $0 \to T \to A \to F \to 0$ induced by the torsion pair $(\mathcal{T}, \mathcal{F})$. Then $T \in {^\perp{C}}$ in $\mathcal{A}$ since $C$ has injective dimension at most one and 
\[ \hom_\mathcal{H}(\Sigma^{-1}T, C) \simeq \ext^1_\mathcal{A}(T, C) = 0 \]
thanks to Lemma~\hyperref[HRSExtTF]{\ref*{lem:HRS-exts}~(\ref*{HRSExtTF})}. Since $C$ is assumed to be a cogenerator in $\mathcal{H}$, we have $T = 0$ and $A\simeq F \in \mathcal{F}$.

Finally, we show that also $\cogen({C}) = \mathcal{F}$.
It follows from Lemma~\ref{lem:cotilting-TF} that $\cogen{(C)}$ is a torsion-free class in $\mathcal{A}$ with the torsion class given by
\[ \mathcal{T}' = \{ T \in \mathcal{A} \mid \hom_\mathcal{A}(T, C) = 0 \} \]
Suppose now that $A \in \mathcal{F}$ and consider an exact sequence $0 \to T' \to A \to F' \to 0$ induced by the torsion pair $(\mathcal{T}', \cogen({C}))$. Then $T'\in \mathcal{F}$ and
\[ \hom_\mathcal{H}(T', C) \simeq \hom_\mathcal{A}(T', C) = 0 \]
thanks to Lemma~\hyperref[HRSExtFF]{\ref*{lem:HRS-exts}~(\ref*{HRSExtFF})}. Since $C$ is assumed to be a cogenerator in $\mathcal{H}$, we have $T' = 0$ and $A\simeq F' \in \cogen({C})$.
\end{proof}

We conclude the section with a theorem which explains the role of cotilting objects in Grothendieck categories from the point of view of derived equivalences and which generalizes results from~\cite{ColpiTiltingGrothendieck}, \cite[\S3]{ColpiHeart} and~\cite[\S5]{ParraSaorin}.

\begin{thm}\label{thm:maintheorem}
Let $\mathcal{A}$ be a Grothendieck category, $(\mathcal{T}, \mathcal{F})$ be a torsion pair in $\mathcal{A}$ such that $\mathcal{F}$ contains a generator and let $\mathcal{H}$ be the HRS-tilted abelian category. Then the following are equivalent:
\begin{enumerate}[(1)]
\item\label{HRSTiltGrothendieck}{$\mathcal{H}$ is a Grothendieck category.}
\item\label{HRSTiltInjCogen}{$\mathcal{H}$ has an injective cogenerator.}
\item\label{HRSTiltCotilting}{$\mathcal{F} = \cogen({C}) = {^\perp C}$ is associated with a cotilting object $C \in \mathcal{A}$.}
\end{enumerate}
\end{thm}

\begin{proof}
\hyperref[HRSTiltGrothendieck]{(\ref*{HRSTiltGrothendieck})}$\Rightarrow$\hyperref[HRSTiltInjCogen]{(\ref*{HRSTiltInjCogen})}
This is well known, see for instance~\cite[Corollary X.4.3]{Stenstrom}.

\hyperref[HRSTiltInjCogen]{(\ref*{HRSTiltInjCogen})}$\Rightarrow$\hyperref[HRSTiltCotilting]{(\ref*{HRSTiltCotilting})}
This is immediate from Proposition~\hyperref[HRSInjCog]{\ref*{prop:CotiltingVSInjective}~(\ref*{HRSInjCog})}.

\hyperref[HRSTiltCotilting]{(\ref*{HRSTiltCotilting})}$\Rightarrow$\hyperref[HRSTiltGrothendieck]{(\ref*{HRSTiltGrothendieck})}
We will use a similar argument as in the proof of~\cite[Theorem 6.2]{St3}. We know from Proposition~\ref{prop:CotiltingVSInjective} and Lemma~\ref{lem:prod-C} that the category of injective objects of $\mathcal{H}$ is equivalent to $\Prodcl({C})$. Furthermore, it is a standard fact that if $\mathcal{H}$ and $\mathcal{H}'$ are two abelian categories with enough injective objects and the corresponding full subcategories of injective objects are equivalent, then also $\mathcal{H} \simeq \mathcal{H}'$. This follows for instance from \cite[Proposition IV.1.2]{AuslanderReitenSmalo}, where an argument is given for the dual situation of abelian categories with enough projective objects.

We will now construct a Grothendieck category $\mathcal{H}'$ with the full subcategory of injective objects equivalent to $\Prodcl{(C)}$. It will follow from the above that $\mathcal{H}\simeq\mathcal{H}'$ is also a Grothendieck category.

To this end, let $G\in\mathcal{A}$ be a generator, $R=\ehom_\mathcal{A}(G)$ and $C' = \hom_\mathcal{A}(G, C) \in \Mod{R}$. Then $C'$ is a pure-injective $R$-module by Theorem~\ref{thm:CotiltingPureInjective} and Proposition~\ref{prop:PureInjChar}. Moreover, it follows from Proposition~\ref{prop:GabrielPopescu} that $H$ induces an equivalence
\[ \Prodcl({C}) \simeq \Prodcl({C'})\;. \]

Let now $\mathcal{B}$ be the category of all additive functors ${R}\modl \to \Ab$, where ${R}\modl$ is the category of finitely presented left $R$-modules. This is a locally coherent Grothendieck category and the functor
\begin{align*}
T\colon \Mod{R} &\longrightarrow \mathcal{B}, \\
M &\longmapsto (M\otimes_R-)|_{R\modl},
\end{align*}
is fully faithful, preserves products and sends pure-injective modules to injective objects of $\mathcal{B}$; see~\cite[Theorem B.16]{JensenLenzing}. In particular, if we put $C'' = T(C') \in \mathcal{B}$, we have an equivalence
\[ \Prodcl({C}) \simeq \Prodcl({C''})\;. \]

As now $C''\in\mathcal{B}$ is an injective object, we have a hereditary torsion pair $(\mathcal{T}', \mathcal{F}')$ in $\mathcal{B}$, where $\mathcal{T}' = \ker\hom_\mathcal{B}(-, C'')$ and $\mathcal{F}' = \cogen{(C'')}$. By definition, $\Prodcl({C''})$ is precisely the class of torsion-free injective objects in $\mathcal{B}$ with respect to this torsion pair. Now we can take the Gabriel quotient $\mathcal{H}' = \mathcal{B}/\mathcal{T}$ (see \cite[\S III]{Gabriel} or~\cite[\S11.1]{Prest}). This is a Grothendieck category whose category of injective objects is equivalent to $\Prodcl({C})$ by \cite[Proposition III.4.9 and Corollaire III.3.2]{Gabriel} or \cite[\S\S X.1 and X.2]{Stenstrom}.
\end{proof}

\section{Torsion pairs in categories of sheaves}
\label{sec:torsion}

Now we aim at specializing to categories of quasi-coherent sheaves on Noetherian schemes.
This section is devoted to the description of suitable torsion pairs in the categories of coherent and quasi-coherent sheaves. This is in view of Theorem~\ref{thm:ClassificationViaTP} a key step for a classification of cotilting classes in $\qcoh_X$, but Theorem~\ref{thm:charTPCoh} on torsion pairs in the category of coherent sheaves seems to be of interest on its own.

For a Noetherian scheme $X$, there is a standard classification result for hereditary torsion pairs in $\qcoh_X$ due to Gabriel. We recall this result now, along with a proof. The reason for including the proof, which is in some aspects more direct than the original one, is twofold: it allows us to describe the corresponding torsion-free classes more directly and some parts of the argument will be useful in the subsequent discussion.

Throughout this section, we make use of the results from Appendix~\ref{sec:AssPoints}, where the classification of injective quasi-coherent sheaves on a Noetherian scheme is summarized and the theory of associated points of quasi-coherent sheaves is recalled. Here let us just introduce the following notation: for a point $x$ on a Noetherian scheme $X$, denote by $\mathscr{J}(x)$ the unique injective indecomposable quasi-coherent sheaf whose support is $\overline{\{x\}}$ (see the beginning of \ref{inj} for a detailed construction).

\begin{deff}\label{def:SpecClosed}
	Let $X$ be a topological space. A subset $Y \subseteq X$ is \emph{specialization closed} if $\overline{\{y\}} \subseteq Y$ for every $y \in Y$. Alternatively, a set is specialization closed if it is a union of closed subsets.
\end{deff}

\begin{prop}[{\cite[\S VI.2]{Gabriel}}]\label{prop:TYFY}
	Let $Y \subseteq X$ be a specialization closed subset. Define 
	\begin{align*}
	\mathcal{T}(Y)=&\{\mathscr{T} \in \qcoh_X\mid \supp \mathscr{T} \subseteq Y\}, \\
	\mathcal{F}(Y)=&\{\mathscr{F}\in \qcoh_X\mid \ass \mathscr{F} \cap Y=\emptyset\}.
	\end{align*}
	Then the pair $(\mathcal{T}(Y), \mathcal{F}(Y))$ is a hereditary torsion pair in $\qcoh_X$.
	
	Moreover, there is a bijective correspondence between hereditary torsion pairs in $\qcoh_X$ and specialization closed subsets of $X$, given by the assignments
	\[(\mathcal{T}, \mathcal{F}) \mapsto \supp \mathcal{T}
	\quad\text{ and }\quad
	Y \mapsto (\mathcal{T}(Y), \mathcal{F}(Y)).\]
\end{prop}

\begin{proof}
	The class $\mathcal{T}(Y)$ is closed under arbitrary direct sums, subobjects, quotients and extensions by Corollary~\hyperref[AS4]{\ref*{cor:AssSvazkuBasic}~(\ref*{AS4})} and~\hyperref[AS5]{(\ref*{AS5})}. It follows from Remark~\hyperref[TP2]{\ref*{rem:TP}~(\ref*{TP2})} that $\mathcal{T}(Y)$ is a hereditary torsion class in $\qcoh_X$.
	
	Next we show that $\hom_X(\mathcal{T}(Y), \mathcal{F}(Y))=0.$ Indeed, if $f\colon \mathscr{T} \longrightarrow \mathscr{F}$ is a morphism with $\mathscr{T} \in \mathcal{T}(Y), \; \mathscr{F} \in \mathcal{F}(Y),$ then $\im f \in \mathcal{T}(Y)\cap\mathcal{F}(Y)$. In particular, $\ass \im f \subseteq \supp \im f \subseteq Y$ and $\ass \im f \subseteq \ass \mathscr{F}\subseteq X \setminus Y$, hence $\ass \im f=\emptyset$. It follows by Corollary~\hyperref[AS3]{\ref*{cor:AssSvazkuBasic}~(\ref*{AS3})} that $\im f=0$.
	
	To prove that $(\mathcal{T}(Y), \mathcal{F}(Y))$ is a torsion pair, it remains to show that if $\mathscr{F}$ is a quasi-coherent sheaf on $X$ such that $\hom_X(\mathscr{T}, \mathscr{F})=0$ for all $\mathscr{T}\in \mathcal{T}(Y),$ then $\mathscr{F} \in \mathcal{F}(Y)$. Equivalently, we must show that whenever $\mathscr{F}$ is a quasi-coherent sheaf with $\ass \mathscr{F} \cap Y \neq \emptyset,$ there is a quasi-coherent sheaf $\mathscr{T}$ with $\supp \mathscr{T} \subseteq Y$ and a nonzero morphism $\mathscr{T}\longrightarrow\mathscr{F}$. However, this is immediate from Proposition~\ref{prop:TestingSheaf}.
	
	Clearly for any hereditary torsion pair $(\mathcal{T}, \mathcal{F})$ in $\qcoh_X$, the set $\supp \mathcal{T}$ is specialization closed (if $\mathscr{F}$ is a quasi-coherent sheaf with $\mathscr{F}_x\neq 0$ and $y \in \overline{\{x\}},$ then $\mathscr{F}_x=\mathscr{F}_y\otimes_{\oh_{X,y}}\oh_{X,x}$, so $\mathscr{F}_y \neq 0$). This shows that both the assignments in the statement are well-defined.
	
	If we start with a specialization closed subset $Y \subseteq X$, then clearly $\supp \mathcal{T}(Y)\subseteq Y$. Since for any $x\in Y$, there exists a coherent sheaf $\mathscr{F}$ with $\supp\mathscr{F} = \overline{\{x\}}$, we in fact have $\supp \mathcal{T}(Y)= Y$.
	
	It remains to prove that given any hereditary torsion pair in $(\mathcal{T}, \mathcal{F})$ in $\qcoh_X$ and $Y=\supp\mathcal{T}$, we have that $\mathcal{T}=\mathcal{T}(Y)$ and $\mathcal{F}=\mathcal{F}(Y)$. Clearly $\mathcal{T} \subseteq \mathcal{T}(Y)$, so by Lemma~\ref{lem:TorsionPairsInclusion}, it suffices to show that $\mathcal{F} \subseteq \mathcal{F}(Y)$. In other words, we must prove that $Y\cap\ass \mathcal{F} = \emptyset$. Suppose that this is not the case, that is, that there exists $x \in Y$ and $\mathscr{F}\in\mathcal{F}$ with $x \in \ass\mathscr{F}$. As $\mathcal{F}$ is closed under injective envelopes by \cite[VI.3.2]{Stenstrom}, we can use Corollary~\ref{cor:AssF=AssEF} to replace $\mathscr{F}$ by $E(\mathscr{F})$. Since $\mathscr{J}(x)$ is a direct summand of $E(\mathscr{F})$ by Lemma~\ref{lem:AssOfJx}, we also have $\mathscr{J}(x) \in \mathcal{F}$. However, if we choose any $\mathscr{T}\in \mathcal{T}$ with $\mathscr{T}_x \neq 0$, then
	%
\begin{multline*}
\mathrm{Hom}_X\left(\mathscr{T}, \mathscr{J}(x)\right)=
\mathrm{Hom}_X\left(\mathscr{T}, i_*\left(\widetilde{E_x}\right)\right) \simeq \\ \simeq
\mathrm{Hom}_{\spec \oh_{X,x}}\left(i^*\left(\mathscr{T}\right), \widetilde{E_x}\right) \simeq
\hom_{\oh_{X,x}}(\mathscr{T}_x, E_x) \ne 0\;,
\end{multline*}
	since $E_x$ is an injective cogenerator of $\Mod \oh_{X,x}$. This contradicts the assumption that $(\mathcal{T}, \mathcal{F})$ is a torsion pair and finishes the proof.
\end{proof}

\begin{cor}\label{cor:limFY=FY}
	Let $X$ be a Noetherian scheme. Then every hereditary torsion-free class in $\qcoh_X$ is closed under direct limits.
\end{cor}

\begin{proof}
In view of Proposition~\ref{prop:TYFY}, it is enough to show that the classes of the form $\mathcal{F}(Y)$ (for $Y$ specialization closed) are closed under direct limits. This is, however, an immediate consequence of Lemma~\ref{lem:AssFlat}.
\end{proof}

Now we easily obtain an analogous classification of hereditary torsion pairs in $\coh_X$.

\begin{prop}\label{prop:T0YF0Y}
For a Noetherian scheme $X$, there is a bijective correspondence between hereditary torsion pairs in $\coh_X$ and specialization closed subsets $Y \subseteq X$, given by the assignments
	\[(\mathcal{T}_0, \mathcal{F}_0) \mapsto \supp \mathcal{T}_0
	\quad\text{ and }\quad
	Y \mapsto (\mathcal{T}_0(Y), \mathcal{F}_0(Y)),\]	
where $\mathcal{T}_0(Y)=\{\mathscr{T} \in \coh_X\mid \supp \mathscr{T} \subseteq Y\}$ and $\mathcal{F}_0(Y)=\{\mathscr{F}\in \coh_X\mid \ass \mathscr{F} \cap Y=\emptyset\}$.
\end{prop}

\begin{proof}
In view of Proposition~\ref{prop:TYFY} and Lemma~\ref{lem:TorsionPairsQcohvsCoh}, it only remains to observe that a torsion pair $(\mathcal{T}_0, \mathcal{F}_0)$ in $\coh_X$ is hereditary if and only if $(\mathcal{T}, \mathcal{F}) = (\Limcl\mathcal{T}_0, \Limcl\mathcal{F}_0)$ is hereditary in $\qcoh_X$. To this end, if $\mathcal{T}$ is closed under taking subsheaves, so is clearly $\mathcal{T}_0 = \mathcal{T}\cap \coh_X$. On the other hand, each $\mathscr{T} \in \mathcal{T}$ is a direct union $\mathscr{T} = \bigcup_{i\in I}\mathscr{T}_i$ of its coherent subsheaves belonging to $\mathcal{T}_0$. If $\mathcal{T}_0$ is closed under taking subsheaves and $\mathscr{S}\subseteq\mathscr{T}$, then $\mathscr{S} =  \bigcup_{i\in I} (\mathscr{S}\cap\mathscr{T}_i) \in \mathcal{T}$.
\end{proof}

If $X$ is Noetherian and affine, then in fact each torsion pair in $\coh_X$ is hereditary by~\cite[Proposition 2.5]{St1}, so we have a full classification of all torsion pairs in this case. In the non-affine situation, there may exist non-hereditary torsion pairs.

\begin{example} \label{example:NonHeredTorsion}
	Let $X = \mathbb{P}^1_k = \proj{k[x_0, x_1]}$ be a projective line over a field and let $\mathcal{T}_0\subseteq\coh_X$ be the smallest class containing the structure sheaf $\oh_X$ and closed under extensions and quotients. This is a torsion class by Remark~\hyperref[TP2]{\ref*{rem:TP}~(\ref*{TP2})} and $\oh(-1)$ belongs to the corresponding torsion-free class $\mathcal{F}_0$ since it has no global sections. However, there is an inclusion $\oh(-1)\hookrightarrow\oh_X$, so $\mathcal{T}_0$ is not hereditary.
\end{example}

The reason is that hereditary torsion pairs have a very geometric meaning. To elucidate this, we give the following definition, which encodes a natural compatibility of the notion of torsion pairs and the underlying geometry.

\begin{deff}\label{def:TP-local}
Let $X$ be a Noetherian scheme. We say that a torsion pair $(\mathcal{T}, \mathcal{F})$ in $\coh_X$ is \emph{locally compatible} if a stronger version of Definition~\hyperref[TPDef1]{\ref*{def:torznipar}~(\ref*{TPDef1})} holds: $\homSh_X(\mathcal{T}, \mathcal{F})=0$.
\end{deff}

\begin{rem}
If $X = \spec R$ is affine, then any torsion pair is locally compatible since then $\homSh_X(\widetilde{M}, \widetilde{N})\simeq \widetilde{\hom_R}(M, N)$ for each pair $M, N$ of finitely generated $R$-modules. On the other hand, the torsion pair from Example~\ref{example:NonHeredTorsion} is not locally compatible since $\oh_X(U) \simeq \oh(-1)(U)$ for each open affine subset $U \subseteq \mathbb{P}^1_k$.
\end{rem}

Using the adjunction between $\otimes$ and $\homSh_X$, one can reformulate local compatibility to another condition, which was used for instance in Thomason's classification~\cite{Thomason} of localizing subcategories in the perfect derived category of~$X$.

\begin{deff}\label{def:tensorIdeal}
Let $\mathcal{X}\subseteq\coh_X$ be a full additive subcategory. Then $\mathcal{X}$ is a \emph{tensor ideal} if $\mathscr{G}\otimes\mathscr{F}\in\mathcal{X}$ for each $\mathscr{G}\in\coh_X$ and $\mathscr{F}\in\mathcal{X}$.
\end{deff}

\begin{lem}\label{lem:tensorIdealLocComp}
Let $X$ be a Noetherian scheme. Then a torsion pair $(\mathcal{T}_0, \mathcal{F}_0)$ in $\coh_X$ is locally compatible if and only if $\mathcal{T}_0$ is a tensor ideal. 
\end{lem}

\begin{proof}
This is purely formal. Suppose that the torsion pair is locally compatible and let $\mathscr{T}\in\mathcal{T}_0 = \ker\homSh_X(-, \mathcal{F}_0)$ be a torsion object and $\mathscr{G}\in\coh_X$ any coherent sheaf. Then, using the standard adjunction, we have
\[ \hom_X(\mathscr{G}\otimes\mathscr{T}, \mathscr{F}) \simeq \hom_X\left(\mathscr{G}, \homSh(\mathscr{T}, \mathscr{F})\right) = 0 \]
for each $\mathscr{F}\in\mathcal{F}_0$. Thus $\mathscr{G}\otimes\mathscr{T} \in \mathcal{T}_0$.

For the 'if' part, suppose there exist $\mathscr{T}\in\mathcal{T}_0$ and $\mathscr{F}\in\mathcal{F}_0$ with $\homSh_X(\mathscr{T}, \mathscr{F}) \ne 0$. If we put $\mathscr{G} = \homSh_X(\mathscr{T}, \mathscr{F})$, there is certainly a non-zero morphism
\[ \mathscr{G} \longrightarrow \homSh_X(\mathscr{T}, \mathscr{F}). \]
By the adjunction again, we obtain a non-zero morphism
\[ \mathscr{G} \otimes \mathscr{T} \longrightarrow \mathscr{F}, \]
hence the torsion class $\mathcal{T}_0$ is not a tensor ideal.
\end{proof}

In the sequel, we put more emphasis on the torsion-free class $\mathcal{F}_0$ rather than on the torsion class $\mathcal{T}_0$. We cannot in general expect $\mathcal{F}_0$ to be a tensor ideal since it is typically not closed under taking cokernels, but quite often it turns out to be closed under tensoring by line bundles. In that context, the following definition is useful.

\begin{deff} \label{def:AmpleFamily}
\begin{enumerate}[(1)]
\item{We say that a scheme $X$ has \emph{an ample family of line bundles} if there are global sections $f_i$ of line bundles $\mathscr{L}_i, \; i \in I$, such that the sets $D(f_i)=\{x \in X \mid f_i(x)\neq 0\}, \; i \in I$ form an affine open cover of $X$.}
\item{We say that a scheme $X$ \emph{has the resolution property} if every coherent sheaf is an epimorphic image of a vector bundle (i. e. a locally free sheaf of finite rank).}
\end{enumerate}
\end{deff}

A Noetherian scheme $X$ which has an ample family of line bundles has the resolution property, which was proved by S.~Kleiman and M.~Borelli in \cite{Borelli}, and independently by L.~Illusie in \cite{Illusie}. In such a case, every coherent sheaf is in fact a factor of a finite direct sum of negative tensor powers of the line bundles $\mathscr{L}_i$. Since any quasi-coherent sheaf is a direct union of its coherent subsheaves, the negative tensor powers of the line bundles $\mathscr{L}_i$ form a set of generators for $\qcoh_X$ as well. The above properties are satisfied for a large class of Noetherian schemes, e.g. for quasi-projective schemes over affine schemes. See \cite[section~2.1]{T-T} for more detailed discussion.

Now we can state and prove the main result of the section.

\begin{thm}\label{thm:charTPCoh}
Let $X$ be a Noetherian scheme and $(\mathcal{T}_0, \mathcal{F}_0)$ be a torsion pair in $\coh_X$. Then the following are equivalent:
\begin{enumerate}[(1)]
\item{$(\mathcal{T}_0, \mathcal{F}_0)$ is hereditary,}
\item{$(\mathcal{T}_0, \mathcal{F}_0)$ is locally compatible (Definition~\ref{def:TP-local}),}
\item{$\mathcal{T}_0$ is a tensor ideal (Definition~\ref{def:tensorIdeal}).}
\end{enumerate}

If, moreover, $X$ has an ample family of line bundles, these are further equivalent to
\begin{enumerate}[(1)]
\setcounter{enumi}{3}
\item{$\mathscr{L}\otimes\mathscr{F} \in \mathcal{F}_0$ for each line bundle $\mathscr{L}$ and each $\mathscr{F}\in \mathcal{F}_0$.}
\end{enumerate}
\end{thm}

\begin{proof}
(1)$\Rightarrow$(2)
If $(\mathcal{T}_0, \mathcal{F}_0)$ is hereditary, it is of the form $(\mathcal{T}_0(Y), \mathcal{F}_0(Y))$ for a specialization closed $Y\subseteq X$ by Proposition~\ref{prop:T0YF0Y}, and hence clearly locally compatible since both the torsion and torsion-free class are defined by conditions on stalks.

(2)$\Rightarrow$(1)
Let us assume that $(\mathcal{T}_0, \mathcal{F}_0)$ is locally compatible and put $Y=\supp\mathcal{T}_0$.

We first claim that $Y \cap \ass\mathcal{F}_0 = \emptyset$. Indeed, suppose that we have $\mathscr{F}\in\mathcal{F}_0$ and $x\in\ass\mathscr{F}$. If $\mathscr{T}\in\mathcal{T}_0$, then
\[ \hom_{\oh_{X,x}}(\mathscr{T}_x, \mathscr{F}_x) \cong \homSh_X(\mathscr{T}, \mathscr{F})_x = 0. \]
Since $\kappa(x)\hookrightarrow \mathscr{F}_x$, we also have $\hom_{\oh_{X,x}}(\mathscr{T}_x, \kappa(x)) = 0$, so $\mathscr{T}_x=0$ by the Nakayama lemma. It follows that $x\not\in\supp\mathcal{T}_0$ and the claim is proved.

It follows that $\mathcal{T}_0 \subseteq \mathcal{T}_0(Y)$ and $\mathcal{F}_0 \subseteq \mathcal{F}_0(Y)$. Thus $\mathcal{T}_0 = \mathcal{T}_0(Y)$ by Lemma~\ref{lem:TorsionPairsInclusion}, and it is a hereditary torsion class.

(2)$\Leftrightarrow$(3)
This is just Lemma~\ref{lem:tensorIdealLocComp}.

(3)$\Rightarrow$(4)
If $\mathcal{T}_0$ is a tensor ideal, clearly $\mathscr{L} \otimes \mathcal{T}_0 = \mathcal{T}_0$ for each line bundle $\mathscr{L}$. Since $\mathscr{L} \otimes -$ is an autoequivalence of $\coh_X$, the latter is equivalent to $\mathscr{L} \otimes \mathcal{F}_0 = \mathcal{F}_0$.

(4)$\Rightarrow$(3)
Denote by $\mathscr{L}_i,\; i\in I$ the ample family of line bundles. If (4) holds, we see by the same token that $\mathscr{L}_i^{\otimes n} \otimes \mathcal{T}_0 = \mathcal{T}_0$ for each $i\in I$ and $n\in\mathbb{Z}$. Since each coherent sheaf $\mathscr{G}$ admits a surjection of the form $\mathscr{L}_{i_1}^{\otimes n_1} \oplus \cdots \oplus \mathscr{L}_{i_r}^{\otimes n_r} \to \mathscr{G}$ and since $\mathcal{T}_0$ is closed under quotients and direct sums, it must be a tensor ideal.
\end{proof}

We conclude with a consequence for torsion-free classes in $\qcoh_X$.

\begin{cor} \label{cor:TFQcoh}
Let $X$ be a noetherian scheme with an ample family of line bundles and let $\mathcal{F}\subseteq \qcoh_X$ be a torsion-free class. The following are equivalent:
\begin{enumerate}[(1)]
\item{$\mathcal{F}$ is closed under injective envelopes,}
\item{$\mathcal{F}$ is closed under direct limits and $\mathscr{L}\otimes\mathscr{F} \in \mathcal{F}$ for each line bundle $\mathscr{L}$ and $\mathscr{F}\in \mathcal{F}$.}
\end{enumerate}
\end{cor}

\begin{proof}
Note that if (1) holds, the corresponding torsion pair $(\mathcal{T}, \mathcal{F})$ is hereditary by \cite[VI.3.2]{Stenstrom} and thus, $\mathcal{F}$ is closed under direct limits by Proposition~\ref{prop:TYFY} and Corollary~\ref{cor:limFY=FY}. In particular, the torsion pair is determined in both (1) and (2) by its restriction to $\coh_X$; see Lemma~\ref{lem:TorsionPairsQcohvsCoh}. Now we just apply (1)$\Leftrightarrow$(4) from Theorem~\ref{thm:charTPCoh}. 
\end{proof}

\section{Classification of cotilting sheaves}
\label{sec:classif}

Let $X$ be a fixed Noetherian scheme. The goal is to classify those cotilting torsion-free classes in $\qcoh_X$ which are closed under taking injective envelopes or, equivalently for those $X$ which have an ample family of line bundles, under tensoring with line bundles (recall Theorem~\ref{thm:CotiltingPureInjective} and Corollary~\ref{cor:TFQcoh}).

We already know from Theorem~\ref{thm:ClassificationViaTP} and Propositions~\ref{prop:TYFY} and~\ref{prop:T0YF0Y} that a hereditary torsion-free class $\mathcal{F} = \mathcal{F}(Y)$ in $\qcoh_X$ (where $Y\subseteq X$ is specialization closed) is associated with a cotilting sheaf $\mathscr{C}_Y$ if and only if it contains a generator. Here we discuss under what conditions on $Y$ the class $\mathcal{F}(Y)$ actually contains a generator as well as the relation of the cotilting objects here to the construction of cotilting modules in~\cite{St2} in the affine case. 

In direct analogy with the affine case, it turns out that $\mathcal{F}(Y)$ is generating if and only if the set $Y$ does not contain any associated point of the scheme $X$.  The following lemma proves necessity of this condition: 

\begin{lem} \label{lem:AssGeneratorMin}
	Let $X$ be a Noetherian scheme and $\mathscr{G}$ a generator for $\qcoh_X$. Then $\ass \oh_X \subseteq \ass\mathscr{G}$.
\end{lem}

\begin{proof}
    Consider an epimorphism $e:\mathscr{G}^{\oplus I} \rightarrow \oh_X$, where $I$ is some set. For a point $x \in X$, the induced epimorphism of $\oh_{X,x}$-modules $\mathscr{G}_x^{\oplus I} \rightarrow \oh_{X,x}$ splits and thus, we have that
	$$\ass_{\oh_{X,x}}\oh_{X,x} \subseteq \ass_{\oh_{X,x}}\mathscr{G}_x^{\oplus I}=\ass_{\oh_{X,x}}\mathscr{G}_x.$$
	
	Now $x \in \ass \oh_X$ means that $\mathfrak{m}_x\in \ass_{\oh_{X,x}}\oh_{X,x},$ hence in this case we have $\mathfrak{m}_x \in \ass_{\oh_{X,x}}\mathscr{G}_x$ and thus, $x \in \ass \mathscr{G}$.
\end{proof}

Next we prove that the condition is also sufficient, by constructing a generator $\mathscr{G}$ for $\qcoh_X$ with $\ass \mathscr{G}=\ass \oh_X$. 

\begin{prop} \label{prop:TorFreeGen}
Let $X$ be a Noetherian scheme. Then $\qcoh_X$ admits a generator $\mathscr{G}$ with $\ass \mathscr{G}=\ass \oh_X$.
\end{prop}

\begin{proof}
Fix an affine open cover $X=U_1 \cup U_2 \cup \dots \cup U_k$ of $X$, and quasi-coherent ideals $\mathscr{I}_i \subseteq \oh_X$ corresponding to the closed subschemes $Z_i=X \setminus U_i$ (e.g. taken with the reduced closed subscheme structure), $i=1, 2, \dots, k$.

Given a coherent sheaf $\mathscr{F}$ and a section $s \in \mathscr{F}(U_i)$, by Lemma~\ref{lem:PowersOfI}, there is a positive integer $n$ and a morphism $\varphi: \mathscr{I}_i^{n} \rightarrow \mathscr{F}$ such that $s$ is in the image of $\varphi$. Summing up these contributions (over all $i$ and generating sets of $\mathscr{F}(U_i)$), it follows that there is an epimorphism of the form $\bigoplus_{i, j}\mathscr{I}_i^{n_{ij}}\rightarrow \mathscr{F}$ for some integers $n_{ij} \geq 0$. 

Thus, the sheaf $\mathscr{G}=\bigoplus_{i, n}\mathscr{I}_i^n$ is a generator for $\qcoh_X$. Then $\ass \mathscr{G} \subseteq \ass \oh_X$ folows from the fact that $\mathscr{I}_i^n \subseteq \oh_X$ for all $i$ and $n$, and the converse inclusion follows from Lemma~\ref{lem:AssGeneratorMin}. 
\end{proof}

Let us now discuss how the construction of cotilting objects from Theorem~\ref{thm:ClassificationViaTP} is related to the construction of cotilting modules from~\cite{St2}. Suppose that $Y \subseteq X$ is a specialization closed subset. Denote by $\mathcal{I}(Y)$ the class of all injective quasi-coherent sheaves $\mathscr{E}$ with $\ass \mathscr{E} \cap Y = \emptyset$. That is, $\mathcal{I}(Y)$ consists of all the injectives contained in $\mathcal{F}(Y)$.

\begin{lem}\label{lem:EpiCover}
Suppose that $\ass \oh_X \cap Y = \emptyset$. Then every $\mathcal{F}(Y)$-cover $f\colon \mathscr{F} \to \mathscr{J}$ of an injective quasi-coherent sheaf $\mathscr{J}$ is an epimorphism and it is also an $\mathcal{I}(Y)$-cover.
\end{lem}

\begin{proof}
By Proposition~\ref{prop:TorFreeGen}, we may choose a generator $\mathscr{G} \in \mathcal{F}(Y)$. Then there is an epimorphism $g\colon \mathscr{G}^{\oplus I} \rightarrow \mathscr{J}$ which factors through $f$, which shows that $f$ is an epimorphism.

Since $\mathcal{F}(Y)$ is a hereditary torsion-free class, an injective envelope $i\colon \mathscr{F} \to E(\mathscr{F})$ has $E(\mathscr{F}) \in \mathcal{I}(Y)$. Since $\mathscr{J}$ is injective, there exists a morphism $h\colon E(\mathscr{F}) \to \mathscr{J}$ such that $f = h \circ i$. It is immediate that $h$ is also an $\mathcal{F}(Y)$-precover of $\mathscr{J}$. Since a cover is a retract of any precover (\cite[Lemma 5.8]{G-T}), $\mathscr{F}$ is a direct summand of $E(\mathscr{F})$ and as such $\mathscr{F} \in \mathcal{I}(Y)$.
\end{proof}

Now recall the construction from~\cite{St2} or, to be more precise, a direct generalization of it to non-affine Noetherian schemes.

\begin{constr}\label{constr:ImitatedCotilting}
Fix a specialization closed subset $Y \subseteq X$ that does not contain any associated point of $X$. For each point $y \in Y$, consider an exact sequence

\vspace{0.25cm}
\begin{adjustbox}{max totalsize={1\textwidth}{.9\textheight},center}
\begin{tikzcd}
0 \ar[r]& \mathscr{K}(y) \ar[r, "\alpha_y"]& \mathscr{I}(y) \ar[r, "\beta_y"] & \mathscr{J}(y) \ar[r] & 0,
\end{tikzcd}
\end{adjustbox}
\vspace{0.15cm}

\noindent where $\beta_y\colon \mathscr{I}(y) \longrightarrow \mathscr{J}(y)$ is an $\mathcal{F}(Y)$-cover (recall that $\mathscr{J}(y)$ is the indecomposable injective with support $\overline{\{y\}}$, see \ref{inj}). In particular, $\mathscr{I}(y)$ is injective by Lemma~\ref{lem:EpiCover}.
Define quasi-coherent sheaves
$$\mathscr{K}_Y:=\prod_{y \in Y}\mathscr{K}(y), \quad\text{ and }\quad \mathscr{J}_Y:= \prod_{x \in X \setminus Y}\mathscr{J}(x),$$
and finally, put
$$\mathscr{C}_Y:=\mathscr{K}_Y \oplus \mathscr{J}_Y.$$
\end{constr}

Now we quickly see from Lemma~\ref{lem:Wakamatsu} that $\mathscr{K}(y) \in \mathcal{F}(Y)^\perp$. If we take the product of all the above short exact sequences for all $y\in Y$ together with the trivial short exact sequence $0 \to 0 \to \mathscr{J}_Y \to \mathscr{J}_Y \to 0$, we get an exact (by Proposition~\ref{prop:ExactProd}) sequence

\vspace{0.25cm}
\begin{adjustbox}{max totalsize={1\textwidth}{.9\textheight},center}
	\begin{tikzcd}
		0 \ar[r]& \mathscr{K}_Y \ar[r]& \prod\limits_{y \in Y} \mathscr{I}(y) \oplus \mathscr{J}_Y \ar[r] & \prod\limits_{x \in X} \mathscr{J}(x) \ar[r] & 0,
	\end{tikzcd}
\end{adjustbox}
\vspace{0.15cm}

\noindent
It follows from the proof of Theorem~\ref{thm:CharTiltingClasses} that $\mathscr{C}_Y \oplus \prod_{y \in Y} \mathscr{I}(y)$ is a cotilting sheaf associated with $\mathcal{F}(Y)$. Since $\mathscr{I}(y) \in \Prodcl({\mathscr{C}_Y})$ for all $y\in Y$, it follows that also $\mathscr{C}_Y$ is a cotilting sheaf associated with $\mathcal{F}(Y)$.

\begin{rem}
The arguments for \cite[Theorems 5.3 and 5.4]{St2} generalize in a straightforward way to our setting as well. In particular, the indecomposable sheaves in $\Prodcl(\mathscr{C}_Y)$ are precisely
\[
\mathscr{K}(y), \; y\in Y
\quad\text{ and }\quad
\mathscr{J}(x), \; x\in X\setminus Y\;.
\]

This in particular says that the indecomposable injectives in the HRS-tilted abelian categories $\mathcal{H}$ (Definition~\ref{def:HRS-tilt}) induced by the cotilting objects $\mathscr{C}_Y$ correspond bijectively to the points of $X$.
\end{rem}

We summarize the above discussion as follows.

\begin{thm} \label{thm:classification}
Let $X$ be a Noetherian scheme. Then the assignment
\[ Y \mapsto \mathcal{F}(Y) \]
induces a bijective correspondence between

\begin{enumerate}[(1)]
\item the specialization closed subsets $Y\subseteq X$ such that $\ass \oh_X \cap Y = \emptyset,$ and
\item the hereditary torsion-free classes in $\qcoh_X$ associated with a cotilting sheaf.
\end{enumerate}

If, moreover, $X$ has an ample family of line bundles, then the image of the above correspondence can be also described as

\begin{enumerate}[(1)]
\setcounter{enumi}{2}
\item the torsion-free classes $\mathcal{F}$ in $\qcoh_X$ associated with a cotilting sheaf and such that $\mathscr{L}\otimes\mathcal{F} = \mathcal{F}$ for each line bundle $\mathscr{L}$.
\end{enumerate}
\end{thm}

\begin{proof}
The first part immediately follows as a combination of Theorem~\ref{thm:ClassificationViaTP} and Propositions~\ref{prop:T0YF0Y}  and~\ref{prop:TorFreeGen}, with a description of the corresponding cotilting sheaf given by the above discussion. The second part then follows from Theorem~\ref{thm:CotiltingPureInjective} and Corollary~\ref{cor:TFQcoh}.
\end{proof}

Before we proceed to examining the $1$-dimensional case of the developed theory in more detail, let us draw some consequences of Theorem~\ref{thm:classification} for the class of the (``classically'') torsion-free quasi-coherent sheaves.

Recall that a module $M$ over a Noetherian ring $R$ is called \emph{torsion-free} if for any non-zero divisor $r$ and any $m \in M$, $rm=0$ implies $m=0$. In other words, any zero divisor $r\in R$ of $M$ is a zero divisor of $R$ or, equivalently by~\cite[Proposition 1.2.1]{BrunsHerzog}, any associated prime of $M$ is contained in some associated prime of $R$. If we denote by $Y\subseteq\spec R$ the largest specialization closed subset not containing any associated prime of $R$ (i.e.\ $Y=X \setminus Z$, where $Z$ the set of all $\mathfrak{p}\in\spec R$ that specialize to an associated prime of $R$), it follows that the class of torsion-free modules is precisely the hereditary torsion-free class $\mathcal{F}(Y)$ given by Proposition~\ref{prop:TYFY}.

In order to generalize torsion-freeness to quasi-coherent sheaves on non-affine schemes, the latter description using the associated points is more useful as the classical description using elements is not local in general.

\begin{example}
Suppose that $M$ is a torsion-free $R$-module and $\mathfrak{p}$ is an associated prime of $M$ that is not an associated prime of $R$. Let $\mathfrak{q}_1,\dots\mathfrak{q}_n\in \ass R$ be all the associated primes of $R$ that contain $\mathfrak{p}$ and are minimal such. A particular example of this situation is $R=k[x, y, z]/(xz, yz, z^2)$, $M=R/(y, z)$ and $(y, z)=\mathfrak{p}\subseteq\mathfrak{q}_1=(x, y, z)$.
Picking $f \in (\bigcap_i\mathfrak{q}_i) \setminus \mathfrak{p}$, $M_f$ will not remain torsion-free as an $R_f$-module since $\mathfrak{p}_f$ is not contained in any associated prime of $R_f$. Similarly, by localization at $\mathfrak{p}$ we see that torsion-freeness is not stalk-local either.
\end{example}

\begin{deff}
A quasi-coherent sheaf $\mathscr{F}$ on a Noetherian scheme $X$ is called \emph{torsion-free} if every associated point $x$ of $\mathscr{F}$ is a generization of an associated point of $X$, or equivalently $\mathscr{F} \in \mathcal{F}(Y)$, where $Y\subseteq X$ is the largest specialization closed subset not containing any associated point of $X$.
\end{deff} 

\begin{rem}
Note, however, that the torsion-free condition is affine-local over schemes with no embedded points, in which case the defining condition degenerates into $\ass \mathscr{F} \subseteq \ass \oh_X$. In particular, when $X$ is an integral Noetherian scheme, one recovers the usual (affine-local or stalkwise) definition of torsion-free quasi-coherent sheaves (\cite[Tag 0AVQ]{stacks}). 
\end{rem}


The following is then an immediate consequence of Theorem~\ref{thm:classification}, and generalizes and supplements \cite[Theorem~4.7]{odabasi}.

\begin{cor}\label{cor:TorFreeCovers}
Let $X$ be a Noetherian scheme. The class of torsion-free quasi-coherent sheaves on $X$ is covering and generating.
In fact, it is a cotilting class which is smallest among the cotilting classes closed under injective envelopes.
\end{cor}

\begin{rem}
Let us compare Corollary~\ref{cor:TorFreeCovers} to the case of covering and generating properties of flat sheaves (which are always torsion-free).

A scheme $X$ is called \emph{semi-separated} if the diagonal morphism $\Delta: X \rightarrow X \times_{\mathbb{Z}} X$ is affine, i.e. if the intersection of every pair of affine open subsets of $X$ is again affine. Murfet showed in \cite[Corollary~3.21]{Murfet} that on a semi-separated Noetherian scheme, the class of flat quasi-coherent sheaves is generating, and Enochs and Estrada showed in \cite{EnochsEstrada} that the class is always covering. On the other hand, a recent result in~\cite{SlavikSt} shows that semi-separatedness is also necessary for $\qcoh_X$ to have a flat generator if $X$ is Noetherian. That is, one can take epimorphic flat covers if and only if $X$ is semi-separated.

Corollary~\ref{cor:TorFreeCovers} then shows that in the case of Noetherian schemes that are not necessarily semi-separated, one can still always take at least epimorphic torsion-free covers.
\end{rem}

We conclude the paper by discussing the developed theory in the case of $1$-dimensional schemes in more detail.

\begin{example}\label{example:curves}
Let $X$ be a $1$-dimensional Noetherian scheme, e.g. a quasi-projective curve over a field. 
Fix a specialization closed subset $Y \subseteq X$ avoiding the associated points of $X$. We aim to describe the cotilting sheaf $\mathscr{C}_Y$ from Construction~\ref{constr:ImitatedCotilting} more explicitly.

Fix a point $y \in Y$ and denote by $i\colon \spec \oh_{X,y} \rightarrow X$ the canonical morphism (described in detail in \ref{inj}). Denote by $\widehat\oh_{X,y}$ the $\mathfrak{m}_y$-adic completion of $\oh_{X,y}$, by $j':\spec \widehat \oh_{X,y}\rightarrow \spec \oh_{X,y}$ the map induced by the completion, and set $j(=j_y)=ij'$. 

The local ring $\oh_{X,y}$ satisfies $\mathfrak{m}_y \notin \ass \oh_{X,y}$. In particular, we have 
$$1 \geq \dim \oh_{X,y} \geq \depth \oh_{X,y}\geq 1.$$
It follows that $\oh_{X,y}$ is a local \CM ring and thus, $\widehat\oh_{X,y}$ is a local complete \CM ring, hence it admits a canonical module $\omega_y$ (\cite[Corollary 3.3.8]{BrunsHerzog}). By \cite[Theorem 3.1.17]{BrunsHerzog}, we have $\injdim \omega_y= \depth \widehat\oh_{X,y}=1$. In fact, $\omega_y$ admits an injective resolution

\vspace{0.25cm}
\begin{adjustbox}{max totalsize={1\textwidth}{.9\textheight},center}
	\begin{tikzcd}
\varepsilon:	 &	0 \ar[r] & \omega_y \ar[r] &  E^0 \ar[r, "\pi"] & E^1 \ar[r] & 0\;, & 
	\end{tikzcd}
\end{adjustbox}
\vspace{0.15cm}

\noindent where 
$$E^0=\bigoplus\limits_{\begin{subarray}{c} \mathfrak{p} \in \spec \widehat\oh_{X,y}\\ \height \mathfrak{p}=0\end{subarray}}E(\widehat \oh_{X,y}/\mathfrak{p}) \;\;\text{ and }\;\;E^1=E_{\widehat \oh_{X,y}}(\widehat \oh_{X,y}/ \widehat{\mathfrak{m}}_y) \simeq E_{\oh_{X,y}}(\kappa(y)).$$ 
In particular, $E^0 \in \mathcal{I}(\{\widehat{\mathfrak{m}}_y\})$. Note that all the finite $\widehat \oh_{X,y}$-modules are pure-injective: this follows from \cite[Proposition 4.3.29]{Prest} and the Matlis duality, since every finite $\widehat \oh_{X,y}$-module is its own double dual. 

First we aim to show that $\pi\colon E^0 \rightarrow E^1$ is an $\mathcal{F}(\{\widehat{\mathfrak{m}}_y\})$-cover. Given $M \in \mathcal{F}(\{\widehat{\mathfrak{m}}_y\}),$ we may express $M=\varinjlim_i M_i,$ where $M_i$ runs over all the finite $\widehat\oh_{X,y}$-submodules of $M$. Since $\widehat{\mathfrak{m}}_y \notin \ass M_i,$ the modules $M_i$ are maximal \CM (unless $0$) and thus, $\ext^1_{\widehat \oh_{X,y}}(M_i, \omega_y)=0$ (\cite[Proposition 3.3.3]{BrunsHerzog}). However, since $\omega_y$ is pure-injective, \cite[Lemma 6.28]{G-T} gives
$$\ext^1_{\widehat \oh_{X,y}}(M, \omega_y)=\ext^1_{\widehat \oh_{X,y}}(\varinjlim_i M_i, \omega_y)\simeq \varprojlim_i \ext^1_{\widehat \oh_{X,y}}(M_i, \omega_y)=0.$$ Thus, $\pi$ is an $\mathcal{F}(\{\widehat{\mathfrak{m}}_y\})$-precover. By Proposition~\ref{prop:ElBashir}, the $\mathcal{F}(\{\widehat{\mathfrak{m}}_y\})$-cover of $E(\kappa(y))$ exists, and by \cite[Lemma 5.8]{G-T}, its domain is a direct summand in any precover with a complement in the kernel of the precover. On the other hand, $\omega_y$ is indecomposable since $\ehom_{\widehat\oh_{X,y}}(\omega_y)\simeq \widehat\oh_{X,y}$ is a local ring (\cite[Theorem~3.3.4]{BrunsHerzog}). It follows that $\pi$ is an $\mathcal{I}(\{\widehat{\mathfrak{m}}_y\})$-cover.

For the purposes of this example, let us identify quasi-coherent sheaves over affine schemes $\spec{\oh_{X,y}}, \spec{\widehat \oh_{X,y}}$ with the respective module categories via the global sections functor. Applying $j_{*}$ to $\varepsilon$ in this sense yields a short exact sequence of quasi-coherent sheaves on $X$

\vspace{0.25cm}
\begin{adjustbox}{max totalsize={1\textwidth}{.9\textheight},center}
	\begin{tikzcd}
	0 \ar[r] & j_*(\omega_y) \ar[r] &  j_*(E^0) \ar[r, "j_*(\pi)"] & \mathscr{J}(y) \ar[r] & 0  
	\end{tikzcd}
\end{adjustbox}
\vspace{0.15cm}

\noindent
(note that $j_*$ is exact: both $i, j'$ are affine morphisms, hence $j$ is, therefore the higher direct images $R^ij_*$ vanish, \cite[Tag 073H]{stacks}).

The adjunction $(j^*, j_*)$ yields a commutative diagram 

\vspace{0.25cm}
\begin{adjustbox}{max totalsize={1\textwidth}{.9\textheight},center}
	\begin{tikzcd} 
	\hom_{\widehat \oh_{X,y}}(j^*(\mathscr{F}), E^0) \ar[r, "\simeq "] \ar[d, "\pi\circ -"] & \hom_{X}(\mathscr{F}, j_*(E^0)) \ar[d, "j_*(\pi)\circ -"] \\
	\hom_{\widehat \oh_{X,y}}(j^*(\mathscr{F}), E^1) \ar[r, "\simeq "] & \hom_{X}(\mathscr{F}, j_*(E^1))
	\end{tikzcd}
\end{adjustbox}
\vspace{0.15cm}

for every $\mathscr{F} \in \qcoh_X$ and thus, to show that $j_*(\pi)$ is an $\mathcal{F}(Y)$-precover, it remains to show that $j^*(\mathcal{F}(Y)) \subseteq \mathcal{F}(\{\widehat{\mathfrak{m}}_y\})$. The fact that $i^*(\mathcal{F}(Y)) \subseteq \mathcal{F}(\{\mathfrak{m}_y\})$ is clear from Definition~\ref{def:suppass}, as $i^*$ is identified by the above convention with the stalk functor at $y$. The inclusion $j'^*(\mathcal{F}_0(\{\mathfrak{m}_y\})) \subseteq \mathcal{F}_0(\{\widehat{\mathfrak{m}}_y\})$ amounts to the fact that the $\mathfrak{m}_y$-adic completion preserves maximal \CM modules, and finally, $j'^*(\mathcal{F}(\{\mathfrak{m}_y\})) \subseteq \mathcal{F}(\{\widehat{\mathfrak{m}}_y\})$ follows, since $j'^{*}$ preserves direct limits.

To show that $j_*(\pi)$ is an $\mathcal{F}(Y)$-cover, it is again enough to show that $j_*(\omega_y)$ is indecomposable. Since $i_*$ is fully faithful (by the fact that the counit $i^*i_* \Rightarrow \mathrm{Id}_{\Mod\widehat\oh_{X,y}}$ is an isomorphism), this amounts to showing that $j'_*(\omega_y)$ is indecomposable. Equivalently, one needs to show that $\omega_y$ is indecomposable as an $\oh_{X,y}$-module. Suppose not and consider a non-trivial decomposition $\omega_y=M\oplus N$ as an $\oh_{X,y}$-module. Taking the $\mathfrak{m}_y$-adic completion yields a non-trivial decomposition $\omega_y \simeq \widehat{\omega_y}=\widehat M\oplus \widehat N$ (note that $\widehat M, \widehat N \neq 0$: this follows e.g. by $\mathfrak{m}_y^k \omega_y=\mathfrak{m}_y^k M\oplus \mathfrak{m}_y^k N$ and $\bigcap_k \mathfrak{m}_y^k \omega_y=0$). This contradicts the indecomposability of $\omega_y$ as an $\widehat\oh_{X,y}$-module.

It follows that $j_*(\pi)$ is an $\mathcal{F}(Y)$-cover of $\mathscr{J}(y)$ and hence, we have 
$$\mathscr{K}(y) \simeq j_*(\omega_y)\;.$$
Altogether, the corresponding cotilting sheaf is of the form
$$\mathscr{C}_Y = \prod_{y \in Y} j_{y*}(\omega_y)\times \prod_{x \in X \setminus Y} \mathscr{J}(x)\;.$$
\end{example}

\begin{example}
Let $X = \mathbb{P}^1_k$ be a projective line over a field and $\xi\in X$ be the generic point. Consider a subset $Y \subseteq X$ avoiding $\xi$ (which is automatically spec. closed). Then all the local completed rings $\widehat\oh_{X,y}\simeq k[\![T]\!]$ are further Gorenstein, hence the short exact sequence $\varepsilon$ is now of the form

\vspace{0.25cm}
\begin{adjustbox}{max totalsize={1\textwidth}{.9\textheight},center}
\begin{tikzcd}
0 \ar[r]& \widehat{\oh}_{X,y} \ar[r]& \widehat{\mathscr{Q}}_y \ar[r, "\varphi"]& E_y \ar[r] & 0\;,
\end{tikzcd}
\end{adjustbox}
\vspace{0.15cm}

\noindent where $\widehat{\mathscr{Q}}_y \simeq k(\!(T)\!)$ is the completed fraction field. Thus, we obtain
\[ \mathscr{K}(y) = j_{y*}(\widehat{\oh}_{X,y}),\;\;\; y \in Y. \]
\end{example}

\begin{example} \label{example:P1VsKronecker}
If we consider $X = \mathbb{P}^1_k$ again, there are cotilting sheaves in $\qcoh_X$ whose associated cotilting classes $\mathcal{F}$ are not closed under injective envelopes (equivalently, under twists). Indeed, consider the non-hereditary torsion pair $(\mathcal{T}_0, \mathcal{F}_0)$ in $\coh_X$ from Example~\ref{example:NonHeredTorsion}. The objects of $\mathcal{F}_0$ are precisely finite direct sums of copies of $\oh(n)$, $n<0$. Since $\oh(1)$ is an ample line bundle, $\mathcal{F}_0$ is a generating class in $\coh_X$ and, by Theorem~\ref{thm:ClassificationViaTP}, $\mathcal{F} = \Limcl\mathcal{F}_0$ is a cotilting torsion-free class in $\qcoh_X$.

Let us collect more information about $\mathcal{F}$ to illustrate the theory. First of all, $\mathcal{F}$ consists precisely of (possibly infinite) direct sums of copies of $\oh(n)$, $n<0$. This follows by the same argument as for \cite[Proposition 3.6]{LenzingTransfer}. Furthermore, the torsion pair $(\mathcal{T}, \mathcal{F})$ is split, i.e.\ $\ext^1_\mathcal{A}(\mathcal{F}, \mathcal{T}) = 0$, or equivalently the short exact sequence from Definition~\hyperref[TPDef2]{\ref*{def:torznipar}~(\ref*{TPDef2})} splits for every $A \in \mathcal{A}$. Indeed, using arguments dual to those in Proposition~\ref{prop:CotiltingVSInjective}, one shows that $T = \Sigma^{-1}(\oh_X \oplus \oh(1))$ is a projective generator in the tilted category, so that the tilted category $\mathcal{H}$ is equivalent to $\Mod{R}$, where $R = \ehom_\mathcal{A}(T) = k(\bullet\!\rightrightarrows\!\bullet)$ is the Kronecker algebra (this recovers a special case of Beilinson's equvalences~\cite{Beilinson}). Since $R$ is well-known to be hereditary, Lemma~\hyperref[HRSExtFT]{\ref*{lem:HRS-exts}~(\ref*{HRSExtFT})} implies
\[ \ext^1_\mathcal{A}(\mathcal{F}, \mathcal{T}) \simeq \ext^2_\mathcal{H}(\mathcal{F}, \Sigma^{-1}\mathcal{T}) = 0 \]
(see~\cite[Lemma 2.1 in Chapter~2]{HRS} and also \cite[\S6]{ColpiFuller3} and \cite[Theorem 5.2]{StovicekKernerTrlifaj}).

In particular, $\mathcal{T}\subseteq\mathcal{F}^\perp$. One can check that also $\oh(-1), \oh(-2) \in \mathcal{F}^\perp$ and that $C = \oh(-1) \oplus \oh(-2)$ is a cotilting object cogenerating $\mathcal{F}$. On the other hand, $\oh(n) \notin \mathcal{F}^\perp$ for $n\le -3$ as there exist non-split short exact sequences $0 \to \oh(n) \to \oh(n+1) \oplus \oh(n+1) \to \oh(n+2) \to 0$. Thus,
\[ \mathcal{F}^\perp = \oh(-2) \otimes \mathcal{T} = \{ \mathscr{F} \in \qcoh_X \mid \oh(n) \text{ is not a summand of } \mathscr{F} \text{ for each } n \le -3 \}, \]
and this class has exact products by Proposition~\ref{prop:ExactProd} (in contrast to Example~\ref{example:ExtProducts}).
\end{example}

\appendix

\section{Ext-functors and products in abelian categories}\label{sec:ExtProd}

We start with discussing a few aspects of interaction of the $\ext^1$-functor with infinite products in a general Grothendieck category. This is necessary since, in contrast to module categories, Grothendieck categories do not have exact products in general. That is, given an infinite set $I$, the product functor
$$\prod_{i \in I}\colon \mathcal{A}^{I} \longrightarrow \mathcal{A}$$
preserves kernels, but it may not preserve cokernels. In fact, the exactness fails even for $\qcoh_X$, where $k$ is a field and $X=\mathbb{P}^1_k=\proj{k[x_0, x_1]}$ is the projective line over $k$; see~\cite[Example 4.9]{Krause}.

Let $\mathcal{A}$ be an abelian category and $B_i$, $i \in I$, be a collection of objects in $\mathcal{A}$ such that the product $\prod_{i\in I}B_i$ exists in $\mathcal{A}$. For any additive functor $F$ from $\mathcal{A}$ to the category of abelian groups we obtain a canonical morphism $F(\prod_i B_i) \to \prod_i F(B_i)$ whose components are $F(\pi_i)\colon F(\prod_i B_i) \to F(B_i)$, where $\pi_i\colon \prod_i B_i \to B_i$ are the product projections. Given $A \in \mathcal{A}$, we can of course specialize this to $F = \ext^1_\mathcal{A}(A,-)$:
$$\psi_{A,B_i}\colon \ext^1_\mathcal{A}\Big(A, \prod_{i \in I}B_i\Big) \longrightarrow \prod_{i \in I}\ext^1_\mathcal{A}(A, B_i)\;.$$
Unlike in module categories, this might not be an isomorphism (see Example~\ref{example:ExtProducts} below), but it still is injective. This in fact follows by formally dualizing \cite[Proposition 8.1]{ColpiFuller3}, but we provide a short proof for the reader's convenience.

\begin{prop}\label{prop:ExtMono}
If $\mathcal{A}$ is an abelian category and $A$ and $B_i$, $i \in I$, is a collection of objects such that $\prod_{i\in I}B_i$ exists in $\mathcal{A}$, then the map $\psi_{A,B_i}$ above is injective.
\end{prop}

\begin{proof}
Consider an extension $\varepsilon\colon 0 \to \prod_i B_i \stackrel{\lambda}\to E \to A \to 0$ representing an element of $\ext^1_\mathcal{A}(A, \prod_{i \in I}B_i)$. For each $i \in I$, consider the commutative diagram
%
\begin{adjustbox}{max totalsize={1\textwidth}{.9\textheight},center}
\begin{tikzcd} 
0 \ar[r]  & \prod\limits_{i \in I}B_i \ar[r, "\lambda"]\ar[d, "\pi_i"]& E \ar[r]\ar[d, "\sigma_i"] & A \ar[r]\ar[d, equal] & 0 \\
0 \ar[r]  & B_i \ar[r, "\lambda_i"]& E_i \ar[r]\ar[l, bend left, dotted, "\rho_i"]& A \ar[r] & 0 \\
\end{tikzcd}
\end{adjustbox}

\noindent
where the lower line is the pushout of $\varepsilon$ along the product projection $\pi_i$. If $\psi_{A,B_i}([\varepsilon]) = 0$, then $\lambda_i$ splits for each $i$. Hence there is a retraction $\rho_i\colon E \to B_i$, indicated by the dotted arrow in the diagram. In particular, $\pi_i$ factorizes through $\lambda$ since $\pi_i = \rho_i\sigma_i\lambda$. Using the universal property of the product, we obtain a factorization through $\lambda$ of the identity on $\prod_i B_i$. That is, $\lambda$ splits as well and $\varepsilon$ must represent the zero element of $\ext^1_\mathcal{A}(A, \prod_{i \in I}B_i)$.
\end{proof}

\begin{cor}\label{prop:ExtTriv}
For any object $A \in \mathcal{A}$ and any collection of objects ${B_i \in \mathcal{A}},$ $i \in I$ for which $\prod_{i\in I}B_i$ exists, we have that 
$$\ext_\mathcal{A}^{1}\Big(A, \prod_{i \in I}B_i\Big)=0 \text{ if and only if }\;\forall i \in I: \;\ext_\mathcal{A}^{1}\Big(A, B_i\Big)=0\;. $$
\end{cor}


\begin{cor}\label{cor:BunoProd}
Let $\mathcal{A}$ be an abelian category with products. For any class $\mathcal{S}$ of objects in $\mathcal{A}$ we have $$\ker \ext^1_{\mathcal{A}}(-, \mathcal{S})=\ext^1_{\mathcal{A}}(-, \Prodcl(\mathcal{S})).$$
\end{cor}

An example of $\psi_{A,B_i}$ actually not being an isomorphism can be found in the category of quasi-coherent sheaves on a projective line.

\begin{example}\label{example:ExtProducts}
Suppose that $\mathcal{A}$ is a hereditary Grothendieck category such that products are \emph{not} exact in $\mathcal{A}$. That is, there is an infinite set $I$ and a family of objects $B_i$, $i \in I$, such that the first right derived functor of $\prod_{i\in I}\colon \mathcal{A}^I \to \mathcal{A}$ does not vanish on $(B_i \mid i\in I)$. More specifically, we can chose for each $i\in I$ an injective resolution of $B_i$,

\vspace{0.25cm}
\begin{adjustbox}{max totalsize={1\textwidth}{.9\textheight},center}
\begin{tikzcd} 
0 \ar[r]  & B_i \ar[r]& E^0_i \ar[r]& E^1_i \ar[r] & 0
\end{tikzcd}
\end{adjustbox}
\vspace{0.25cm}

\noindent
and the product

\vspace{0.25cm}
\begin{adjustbox}{max totalsize={1\textwidth}{.9\textheight},center}
\begin{tikzcd} 
0 \ar[r]  & \prod\limits_{i \in I} B_i \ar[r]& \prod\limits_{i \in I} E^0_i \ar[r, "f"]& \prod\limits_{i \in I} E^1_i
\end{tikzcd}
\end{adjustbox}
\vspace{0.25cm}

\noindent
of these resolutions sequences is left, but not right exact. Since $\mathcal{A}$ is hereditary, the image of $f$ is an injective object and, hence, a summand of $\prod_{i \in I} E^1_i$. Let $\iota\colon A{\hookrightarrow}\prod_{i \in I} E^1_i$ be an inclusion of a complement of the image of $f$. Then the compositions $(\pi_i\circ\iota\colon A \to E^1_i \mid i \in I)$ define an element of $\prod_{i \in I}\ext^1_\mathcal{A}(A, B_i)$ which cannot be in the image of $\psi_{A,B_i}$.

To demonstrate a concrete example, let again $k$ be a field and $\mathcal{A} = \qcoh_X$ for $X = \mathbb{P}^1_k$. By~\cite[Example 4.9]{Krause}, we can take $I = \mathbb{N}$ and for each $n \in \mathbb{N}$ we can let $B_n$ be the kernel of the canonical epimorphism
\[ \oh(-n) \otimes_k \hom_X(\oh(-n),\oh) \longrightarrow \oh. \]
\end{example}

\begin{rem}\label{rem:ExtCoproductsOk}
As in~\cite[\S8]{ColpiFuller3}, one can also consider the dual situation, that is, the canonical morphisms
\[ \varphi_{B_i,A}\colon \ext^1_\mathcal{A}\Big(\bigoplus_{i \in I}B_i, A\Big) \longrightarrow \prod_{i \in I}\ext^1_\mathcal{A}(B_i, A) \]
However, these are isomorphisms as long as our abelian category has enough injectives, so in particular for any Grothendieck category. To see that, let

\vspace{0.25cm}
\begin{adjustbox}{max totalsize={1\textwidth}{.9\textheight},center}
\begin{tikzcd} 
{E}^\bullet\colon & E^0 \ar[r]& E^1 \ar[r]& E^2 \ar[r]& E^3 \ar[r]& \cdots
\end{tikzcd}
\end{adjustbox}
\vspace{0.25cm}

\noindent
be an injective resolution of $A$ and consider the following commutative square of abelian groups:

\vspace{0.25cm}
\begin{adjustbox}{max totalsize={1\textwidth}{.9\textheight},center}
\begin{tikzcd} 
{\ext^1_\mathcal{A}\Big(\bigoplus\limits_{i \in I}B_i, A\Big)} \ar[d, "\varphi_{B_i,A}"'] \ar[r, equal]&
{H^1\left(\hom_\mathcal{A}\Big(\bigoplus\limits_{i \in I}B_i, {E}^\bullet\Big)\right)} \ar[d, "\mathrm{can}", "\simeq"']
\\
{\prod\limits_{i \in I}\ext^1_\mathcal{A}(B_i, A)} \ar[r, equal]&
{\prod\limits_{i \in I}H^1\big(\hom_\mathcal{A}(B_i, {E}^\bullet)\big)}
\end{tikzcd}
\end{adjustbox}
\end{rem}

\section{Quasi-coherent sheaves on locally Noetherian schemes}\label{sec:AssPoints}

\subsection{Injective sheaves on locally Noetherian schemes}\label{inj}

This part is devoted to summarizing results about sheaves on locally Noetherian schemes. 

We start with structure of injective quasi-coherent sheaves on a locally Noetherian scheme $X$. It directly generalizes the structure of injective modules over a Noetherian ring; see \cite[Ch.~IV and VI]{Gabriel} and \cite[\S II.7]{H-RD}. Namely, for each point $x \in X$ there is an associated indecomposable injective sheaf $\mathscr{J}(x)$, and every injective sheaf is a direct sum of the sheaves $\mathscr{J}(x)$ for various points $x \in X$. 

In order to define the sheaves $\mathscr{J}(x)$, let us first briefly describe a natural embedding $\spec\oh_{X,x} \rightarrow X$. Consider $x \in X$ and denote by $Y$ the set of all generizations of $x$, i.e. the set of all points $y$ with $x \in \overline{\{y\}}$. Upon fixing an affine open neighbourhood $U$ of $x$, we have that 
$$Y=\{y \in U\mid \mathfrak{p}_y \subseteq \mathfrak{p}_x\},$$
where $\mathfrak{p}_y$ denotes the prime ideal of $\oh_X(U)$
 corresponding to the point $y \in U$. The localization map $\oh_X(U){\rightarrow}\oh_X(U)_{\mathfrak{p}_x}=\oh_{X,x}$ induces an embedding $$\spec{\oh_{X,x}} \stackrel{j}\longrightarrow \spec{{O}_X(U)}=U,$$
and composition of $j$ with the open immersion $U \subseteq X$ yields an embedding of schemes 
$$\spec{\oh_{X,x}}\stackrel{i}{\hookrightarrow}X$$
with $i(\spec{\oh_{X,x}})=Y$. 
It can be checked that $i$ is independent of the choice of $U$. 
 
Define 
$$\mathscr{J}(x)=i_*\left(\widetilde{E_x}\right),$$
where $E_x$ is the injective hull of $\kappa(x)$, the residue field at $x$ (in $\Mod \oh_{X,x}$). Since $\spec{\oh_{X,x}}$ is a Noetherian scheme, it follows that the morphism $i$ is quasi-compact and quasi-separated (cf. \cite[Remark~10.2 (3) and Definition~10.22]{GW}) and that $\mathscr{J}(x)$ is a quasi-coherent sheaf (\cite[Corollary~10.27]{GW}). Let us now state a version of the classification theorem.

\begin{prop}\label{prop:injectives-classification}
Let $X$ be a locally Noetherian scheme and $\mathscr{F}$ a quasi-coherent sheaf. Then the following are equivalent:
\begin{enumerate}[(1)]
\item\label{IC1}{$\mathscr{F}$ is an injective $\oh_X$-module.}
\item\label{IC2}{$\mathscr{F}$ is an injective quasi-coherent sheaf (i.e.\ it is an injective object in $\qcoh_X$).}
\item\label{IC3}{For every $x \in X$, $\mathscr{F}_x$ is an injective $\oh_{X,x}$-module.}
\item\label{IC4}{$\mathscr{F}$ is a direct sum of sheaves of the form $\mathscr{J}(x)$ for various $x \in X$.}
\end{enumerate}
\end{prop}

\begin{proof}
The equivalence of \hyperref[IC1]{(\ref*{IC1})}, \hyperref[IC3]{(\ref*{IC3})} and \hyperref[IC4]{(\ref*{IC4})} is proved in \cite[II.7.17]{H-RD}. Let us briefly comment on the equivalence of \hyperref[IC1]{(\ref*{IC1})} and \hyperref[IC2]{(\ref*{IC2})}.

The implication \hyperref[IC1]{(\ref*{IC1})} $\Rightarrow$ \hyperref[IC2]{(\ref*{IC2})} follows directly from the fact that monomorphisms in $\qcoh_X$ are precisely monomorphisms in $\Mod{\oh_X}$ between objects from $\qcoh_X$.

Conversely, suppose that $\mathscr{F}$ is injective as a quasi-coherent sheaf. By \cite[II.7.18]{H-RD}, there is a monomorphism $\mathscr{F} \hookrightarrow \mathscr{G}$, where $\mathscr{G}$ is a quasi-coherent sheaf which is injective as an $\oh_X$-module. Using the injectivity of $\mathscr{F}$ (in $\qcoh_X$), it follows that this monomorphism splits (in $\qcoh_X$, thus also in $\Mod\oh_X$). Thus, $\mathscr{F}$ is a direct summand of injective $\oh_X$-module, hence it is an injective $\oh_X$-module as well.
\end{proof}

The following consequence of the classification theorem is important for our purposes as well.

\begin{cor}\label{cor:DirLimOfInj}
Let $X$ be a locally Noetherian scheme.
The class of injective quasi-coherent sheaves on $X$ is closed under taking direct limits.
\end{cor}

\begin{proof}
Suppose a sheaf $\mathscr{F}$ is given by 
$$\mathscr{F}=\varinjlim_{i}\mathscr{E}_i$$
with all $\mathscr{E}_i$'s injective quasi-coherent sheaves. 
Consider an arbitrary point $x \in X$. As the stalk functor $(-)_x$ is a left adjoint, we have 
$$\mathscr{F}_x=\varinjlim_i \left(\mathscr{E}_i\right)_x.$$

By Proposition~\ref{prop:injectives-classification}, all the $\oh_{X,x}$-modules $\left(\mathscr{E}_i\right)_x$ are injective. As $\oh_{X,x}$ is a Noetherian ring, it follows that the direct limit $\mathscr{F}_x$ is injective as well. Using Proposition~\ref{prop:injectives-classification} again, we infer that $\mathscr{F}$ is an injective quasi-coherent sheaf.
\end{proof}

A stronger result, which is employed in Section~\ref{sec:torsion}, holds for Noetherian (i.e.\ locally Noetherian and quasi-compact) schemes.

\begin{lem} \label{lem:LocNoethCatg}
Let $X$ be a Noetherian scheme. Then $\qcoh_X$ is a locally Noetherian Grothendieck category, and Noetherian objects are precisely the coherent sheaves.
\end{lem}

\begin{proof}
This was shown in~\cite[Ch.~6, Th\'{e}or\`{e}me 1]{Gabriel}, but we sketch the argument for the reader's convenience. By definition, $X$ has a finite open affine cover $X = U_1 \cup \dots \cup U_n$ such that $\oh_X(U_i)$ are Noetherian rings. If $\mathcal{F}$ is a coherent sheaf on $X$, then $\mathcal{F}(U_i)$ are Noetherian $\oh_X(U_i)$-modules, and hence $\mathcal{F}$ also satisfies the ascending chain conditions on subobjects in $\qcoh_X$. Finally, every quasi-coherent sheaf on $X$ is the union of its coherent subsheaves by~\cite[Exercise II.5.15]{H-AG} or~\cite[Corollary 10.50]{GW}.
\end{proof}

\begin{rem}
Despite the terminology, $\qcoh_X$ need not be a locally Noetherian category if $X$ is a locally Noetherian scheme. An example of this phenomenon was exhibited in \cite[\S II.7, p.~135--136]{H-RD}. Since direct sums of injective quasi-coherent sheaves are still injective by Corollary~\ref{cor:DirLimOfInj}, $\qcoh_X$ then cannot even be a locally finitely generated Grothendieck category due to~\cite[Proposition V.4.3]{Stenstrom}.

To obtain a better feeling, let us have a look at the example from~\cite{H-RD}. It is an integral (i.e.\ reduced and irreducible) scheme $X$ which is a direct union of a chain
\[ U_0 \subsetneq U_1 \subsetneq U_2 \subsetneq U_3 \subsetneq \cdots \]
of open Noetherian subschemes. If we denote by $\mathscr{I}_n$ the quasi-coherent sheaf of ideals of the closed subset $X\setminus U_n$, it is not difficult to check that $\oh_X = \bigcup_n \mathscr{I}_n$ in $\qcoh_X$ (as well as in $\Mod{\oh_X}$). Thus the structure sheaf $\oh_X$ is coherent, but not a finitely generated object of $\qcoh_X$.
\end{rem}

\subsection{Supports and associated points}

Let us now focus more closely on the topic of associated points and points in the support of quasi-coherent sheaves, as these are among the main tools used in this paper. In what follows in this section, let $X$ be a Noetherian scheme.

\begin{deff}\label{def:suppass}
Let $\mathscr{F}$ be a quasi-coherent sheaf on $X$. Define the \emph{support} of the sheaf $\mathscr{F}$ by
$$\supp\mathscr{F}=\{x \in X \mid \mathscr{F}_x \neq 0\}\;.$$

We say that a point $x \in X$ is an \emph{associated point} of $\mathscr{F}$ provided that there is a monomorphism of $\oh_{X,x}$-modules
$$\kappa(x) \hookrightarrow \mathscr{F}_x\;,$$
i.e.\ if there is an (affine) open set $U \subseteq X$ and a section $s \in \mathscr{F}(U)$ such that $\ann_{\oh_{X,x}}(s_x)=\mathfrak{m}_x$. Denote the set of all associated points of $\mathscr{F}$ by $\ass \mathscr{F}$.
\end{deff}

Recall that given a commutative ring $R$, its prime ideal $\mathfrak{p}$ and an $R$-module $M$, we say that $\mathfrak{p}$ is an \emph{associated prime of $M$} if $\mathfrak{p}=\ann(m)$ for some $m \in M$. That is, there is an injection of $R$-modules $R/\mathfrak{p}\hookrightarrow M$ (taking $1+\mathfrak{p}$ to $m$). Denote the set of all associated primes of $M$ by $\ass M$. Similarly, define \emph{support of $M$}, denoted by $\supp M$, as the set of all primes $\mathfrak{p}$ such that $M_{\mathfrak{p}}\neq 0$.

The following lemma describes the basic well-known properties of associated primes. The proof can be found e.g.\ in \cite[Sections~3.1 and 3.2]{Eisenbud}.  

\begin{lem}\label{AssModuluBasic}
Let $R$ be a commutative Noetherian ring, $M$ an $R$-module and $\mathfrak{p}\subseteq R$ a prime ideal. 
\begin{enumerate}[(1)]
\item\label{AM1}{$\mathfrak{p} \in \ass M$ if and only if $\mathfrak{p}_{\mathfrak{p}} \in \ass M_{\mathfrak{p}}.$}
\item\label{AM2}{$\ass M \subseteq \supp M$.}
\item\label{AM3}{$\ass M = \emptyset$ if and only if $\supp M = \emptyset$ if and only if $M=0$.}
\item\label{AM4}{Given any collection $M_i, \; i \in I,$ of $R$-modules, we have 
$$\ass \bigoplus_{i \in I}M_i=\bigcup_{i \in I}\ass M_i\;$$
and
$$\supp \bigoplus_{i \in I}M_i=\bigcup_{i \in I}\supp M_i\;.$$}
\item\label{AM5}{Given a short exact sequence of $R$-modules

\vspace{0.25cm}
\begin{adjustbox}{max totalsize={1\textwidth}{.9\textheight},center}
\begin{tikzcd}
0 \ar[r]& A \ar[r]& B \ar[r] & C \ar[r] & 0\;,
\end{tikzcd}
\end{adjustbox}
\vspace{0.25cm}

we have that $\ass B \subseteq \ass A \cup \ass C$ and $\supp B = \supp A \cup \supp C$.
}
\end{enumerate}
\end{lem}

Now we use these algebraic facts to prove their algebro-geometric counterparts.

\begin{cor}\label{cor:AssSvazkuBasic}
Let $X$ be a Noetherian scheme, $\mathscr{F}$ a quasi-coherent sheaf on $X$ and $x \in X$ a point. 
\begin{enumerate}[(1)]
\item\label{AS1}{$x \in \ass \mathscr{F}$ if and only if there exists an affine open neighbourhood $U$ of $x$ such that $\mathfrak{p}_x \in \ass \mathscr{F}(U)$.}
\end{enumerate}
\begin{enumerate}[(1')]
\item\label{AS1'}{$x \in \ass \mathscr{F}$ if and only if for every affine open neighbourhood $U$ of $x$, \newline{$\mathfrak{p}_x \in \ass \mathscr{F}(U)$.}}\end{enumerate}
\begin{enumerate}[(1)]
\setcounter{enumi}{1}
\item\label{AS2}{$\ass \mathscr{F} \subseteq \supp \mathscr{F}$.}
\item\label{AS3}{$\ass \mathscr{F} = \emptyset$ if and only if $\supp \mathscr{F} = \emptyset$ if and only if $\mathscr{F}=0$.}
\item\label{AS4}{Given any collection $\mathscr{F}_i, \; i \in I,$ of quasi-coherent sheaves on $X$, we have 
$$\ass \bigoplus_{i \in I}\mathscr{F}_i=\bigcup_{i \in I}\ass \mathscr{F}_i\;$$
and $$\supp \bigoplus_{i \in I}\mathscr{F}_i=\bigcup_{i \in I}\supp \mathscr{F}_i\;.$$}
\item\label{AS5}{Given a short exact sequence of quasi-coherent sheaves on $X$

\vspace{0.25cm}
\begin{adjustbox}{max totalsize={1\textwidth}{.9\textheight},center}
\begin{tikzcd}
0 \ar[r]& \mathscr{F} \ar[r]& \mathscr{G} \ar[r] & \mathscr{H} \ar[r] & 0,
\end{tikzcd}
\end{adjustbox}
\vspace{0.25cm}

we have that $\ass \mathscr{G} \subseteq \ass \mathscr{F} \cup \ass \mathscr{H}$ and $\supp \mathscr{G} = \supp \mathscr{F} \cup \supp \mathscr{H}$.
}
\end{enumerate}
\end{cor}

\begin{proof}
If $U$ is an affine open neighbourhood of $x$, the stalk $\mathscr{F}_{x}$ may be computed as $\left(\mathscr{F}(U)\right)_{\mathfrak{p}_x}$ and thus, $x \in \ass \mathscr{F} $ if and only if $\mathfrak{m}_x=\left(\mathfrak{p}_x\right)_{\mathfrak{p}_x} \in \ass \left(\mathscr{F}(U)_{\mathfrak{p}_x}\right)$. Application of Lemma~\hyperref[AM1]{\ref*{AssModuluBasic}~(\ref*{AM1})} thus proves \hyperref[AS1]{(\ref*{AS1})} and \hyperref[AS1']{(\ref*{AS1'}')}. The statement \hyperref[AS2]{(\ref*{AS2})} is clear from the definition. Statements \hyperref[AS3]{(\ref*{AS3})}--\hyperref[AS5]{(\ref*{AS5})} follow directly from its algebraic counterparts using the facts that for any $x \in X,$ the stalk functor $(-)_x$ is exact (to prove \hyperref[AS5]{(\ref*{AS5})}) and preserves direct sums (to prove \hyperref[AS4]{(\ref*{AS4})}). 
\end{proof}

Additionally, let us prove the following lemma on associated points of direct limits of sheaves.

\begin{lem}\label{lem:AssFlat}
Let $X$ be a locally Noetherian scheme and $\mathscr{F}$ a quasi-coherent sheaf on $X$. If $\mathscr{F}$ is a direct limit of a directed system of quasi-coherent sheaves $(\mathscr{F}_i\; | \; i \in I),$ then $$\ass{\mathscr{F}} \subseteq \bigcup_{i \in I}\ass{\mathscr{F}_i}.$$
\end{lem}

\begin{proof}
We start with a reduction to the affine case. Note that for every point $x \in X$, we have that $\mathscr{F}_x=\varinjlim_i (\mathscr{F}_i)_x$ as the stalk functor preserves colimits. It is clearly enough to show the implication
$$ \forall x \in X:\;\; \text{If }  \mathfrak{m}_x \in \ass(\mathscr{F}_x)\text{, then  }\mathfrak{m}_x \in \ass((\mathscr{F}_i)_x) \text{ for some }i \in I,$$
which is easily seen to be a consequence of the affine version of the statement.

Let us now assume that $R$ is a Noetherian commutative ring and $M, M_i, \; i \in I$ are $R$-modules with $\varinjlim_i M_i=M$.
Denote $\nu_i: M_i \rightarrow M$ the canonical homomorphisms and consider $\mathfrak{p} \in \ass{M}.$ Fix an injective homomorphism $R/\mathfrak{p}\stackrel{\iota}\hookrightarrow M$. Since $R$ is Noetherian, the module $R/\mathfrak{p}$ is finitely presented and thus, $\iota$ has a factorization 
$$\iota=\nu_i \iota_i,$$ 
where $i \in I$ is some index and $\iota_i: R/\mathfrak{p} \rightarrow M_i$ is a suitable homomorphism (see e.g.\ \cite[Lemma~2.8]{G-T}). Such $\iota_i$ is necessarily injective since $\iota$ is. It follows that $\mathfrak{p}\in \ass M_i,$ which concludes the proof. \end{proof}

Associated points are closely related to injective envelopes, exactly as in the affine case. In fact, one can define associated points of objects of certain abstract Grothendieck categories this way, \cite[\S IV.2]{Gabriel}. We record the following lemma and its consequence.

\begin{lem}\label{lem:AssOfJx}
For a point $x \in X$, $\ass \mathscr{J}(x)=\{x\}.$
\end{lem}

\begin{proof}
Let $y \in \ass \mathscr{J}(x)$. Note first that $y$ must be a specialization of $x$. Indeed, otherwise there is an open set $U\subseteq X$ which contains $y$, but not $x$. Denoting again by $i\colon \spec{\oh_{X,x}}{\hookrightarrow}X$ the canonical embedding, we would obtain
\[ \mathscr{J}(x)(U) = i_*\left(\widetilde{E_x}\right)(U) = \widetilde{E_x}\left(i^{-1}(U)\right) = \widetilde{E_x}(\emptyset) = 0, \]
which is absurd.

So suppose that $y \in \overline{\{x\}}$. Then $\mathscr{J}(x)_y \cong i_*\left(\widetilde{E_x}\right){\!}_y$, which is just the injective envelope $E_x$ of $\kappa(x)$, viewed as an $\oh_{X,y}$-module. It is well known that $\ass E_x=\{\mathfrak{p}_x\}$, so $\ass \mathscr{J}(x) = \{x\}$.
\end{proof}

\begin{cor}\label{cor:AssF=AssEF}
Let $X$ be a Noetherian scheme and $\mathscr{F} \in \qcoh_X$. Denote by $E(\mathscr{F})$ the injective hull of $\mathscr{F}$. Then $\ass \mathscr{F}= \ass E(\mathscr{F})$.
\end{cor}

\begin{proof}
Obviously, $\ass \mathscr{F} \subseteq \ass E(\mathscr{F})$ as $\mathscr{F}$ is a subsheaf of $E(\mathscr{F})$. 

Suppose for contradiction that there is a point $x \in \ass E(\mathscr{F})\setminus \ass \mathscr{F}$. Then we can express $E(\mathscr{F})$ as a direct sum $\bigoplus_{i\in I}\mathscr{J}(x_i)$ for a certain collection of points $\{x_i \mid i\in I\}$ by Proposition~\hyperref[IC4]{\ref*{prop:injectives-classification}~(\ref*{IC4})}, and clearly $x_i = x$ for some $i$ by Corollary~\hyperref[AS4]{\ref*{cor:AssSvazkuBasic}~(\ref*{AS4})} and Lemma~\ref{lem:AssOfJx}. In particular, $E(\mathscr{F})$ has both $\mathscr{F}$ and $\mathscr{J}(x)$ as subsheaves. However, $\ass \left(\mathscr{F}\cap \mathscr{J}(x)\right) \subseteq \left(\ass \mathscr{F}\right) \cap \{x\}=\emptyset$. Thus, $\mathscr{F}\cap \mathscr{J}(x)=0$, which contradicts the essentiality of the inclusion $\mathscr{F} \subseteq E(\mathscr{F})$. 
\end{proof}

Given a commutative Noetherian ring $R$ and an $R$-module $M$, an associated prime $\mathfrak{p}$ of $M$ can always be ``isolated'' in a finitely generated submodule $N \subseteq M.$ That is, there is a finitely generated submodule $N$ with $\ass N=\{\mathfrak{p}\}$. This is obvious, one simply needs to take $N$ to be an isomorphic copy of $R /\mathfrak{p}$ that is embedded into $M$.

It will be useful to generalize this property for Noetherian schemes. To this end let $x \in X$ be a point and let $Z_x$ be the integral closed subscheme of $X$ with generic point $x$ (such $Z_x$ exists and is unique by \cite[Proposition~3.50]{GW}). If we denote by $j\colon Z_x{\hookrightarrow}X$ the closed immersion, a naive non-affine analogue of $R/\mathfrak{p}$ is the sheaf $j_*(\oh_{Z_x})$. It is coherent as $j_*(\oh_{Z_x}) \cong \oh_X/\mathscr{I}_x$, where $\mathscr{I}_x$ is the coherent sheaf of ideals of $\mathscr{O}_X$ whose sections are precisely those which vanish at $x$ (i.e.\	$\mathscr{I}_x(U) = \mathfrak{p}_x \subseteq \mathcal{O}_X(U)$ is the prime corresponding to $x$ if $x\in U$ and $\mathscr{I}_x(U) = \mathcal{O}_X(U)$ if $x\not\in U$ for an open affine subset $U$ of $X$).

However, in contrast to the affine case, this sheaf does not embed into every quasi-coherent sheaf $\mathscr{F}$ with $x \in \ass \mathscr{F}$. In fact, in general there is no single coherent sheaf $\mathscr{G}$ such that quasi-coherent sheaves $\mathscr{F}$ with $x \in \ass \mathscr{F}$ would be characterized by the existence of a monomorphism $\mathscr{G} \hookrightarrow \mathscr{F}$.

\begin{example}\label{example:TestingSheaf}
Let $X=\mathbb{P}^1_k$ be a projective line over a field $k$ and $\xi\in X$ be the generic point. 
A quasi-coherent sheaf $\mathscr{F}$ with $\xi \in \ass \mathscr{G}$ cannot be torsion, so it contains a line bundle $\mathscr{L}$ as a subsheaf. In fact, any line bundle $\mathscr{L}$ satisfies $\ass \mathscr{L}=\{\xi\}$, and if a testing coherent sheaf $\mathscr{G}$ for $\xi$ existed, a line bundle $\mathscr{L}\subseteq\mathscr{G}$ would also be such a testing sheaf.

However, there is no single line bundle that would embed into any other line bundle, since $\hom_{X}(\oh(m), \oh(n))=0$ if $n<m$.
\end{example}

However, we obtain the following by a straightforward modification of \cite[Tag 01YE]{stacks}.

\begin{prop}\label{prop:TestingSheaf}
Let $X$ be a Noetherian scheme, $x \in X$ be a point, $Z_x$ be the integral closed subscheme of $X$ with generic point $x$ and $j\colon Z_x{\hookrightarrow}X$ be the closed immersion.

Given a quasi-coherent sheaf $\mathscr{F}$, we have $x \in \ass \mathscr{F}$ if and only if there exists a non-zero coherent sheaf of ideals $\mathscr{I} \subseteq \oh_{Z_x}$ such that $j_*(\mathscr{I}) \hookrightarrow \mathscr{F}$. Moreover, in this case $\ass{j_*(\mathscr{I})} = \{x\}$ and $\supp j_*(\mathscr{I}) = \overline{\{x\}}$.\end{prop}

\begin{proof}
We start with the last sentence. Suppose that $0 \ne \mathscr{I} \subseteq \oh_{Z_x}$---then clearly $\supp j_*(\mathscr{I}) \subseteq \overline{\{x\}}$. If $y \in \ass{j_*(\mathscr{I})}$, then $y \in Z_x$, since otherwise $y$ would have an open affine neighbourhood $U$ with $j_*(\mathscr{I})(U) = 0$. If $y \in Z_x$, consider an open affine neighbourhood $U$ of $y$. Then also $x \in U$ and $\ass j_*(\mathscr{I})(U) = \{\mathfrak{p}_x\}$ since $j_*(\mathscr{I})(U)$ is a non-zero ideal of the domain $j_*(\oh_{Z_x})(U)$. This shows that $\ass{j_*(\mathscr{I})} = \{x\}$ and $\supp j_*(\mathscr{I}) \subseteq \overline{\{x\}}$.

Let now $\mathscr{F}$ be a quasi-coherent sheaf with $x \in \ass \mathscr{F}$ and let $\mathscr{I}_x$ be the coherent sheaf of ideals of $Z_x$ as above. We first reduce the situation to the case where $\mathscr{I}_x\cdot \mathscr{F} = 0$. Indeed, let $\mathscr{F}'\subseteq \mathscr{F}$ be the subsheaf of sections annihilated by $\mathscr{I}_x$ (i.e.\ $\mathscr{F}'(U) = \{s\in\mathscr{F}(U) \mid \mathscr{I}_x(U)\cdot s = 0 \}$). Then $\mathscr{F}'$ is quasi-coherent and $x\in \ass{\mathscr{F}'}$ by \cite[Tag 01PO]{stacks}, so we can replace $\mathscr{F}$ by $\mathscr{F}'$.

If $\mathscr{I}_x\cdot \mathscr{F} = 0$, then $\mathscr{F} \cong j_*(j^*\mathscr{F})$ by \cite[Remark 7.35]{GW}. Moreover, since $\mathscr{F}_x \cong j_*(j^*\mathscr{F})_x \cong (j^*\mathscr{F})_x$, we have $x\in\ass{j^*\mathscr{F}}$. If we find a non-zero coherent sheaf of ideals $\mathscr{I} \subseteq \oh_{Z_x}$ such that $\mathscr{I} \hookrightarrow j^*\mathscr{F}$, then by adjunction $j_*(\mathscr{I})$ embeds into $\mathscr{F}$. Hence we reduced the problem to the case where $X = Z_x$.

Assume, therefore, that $X$ is integral, $x$ is its generic point and $x \in \ass{\mathscr{F}}$. Then, by Corollary~\ref{cor:AssSvazkuBasic}, there exists an open affine neighbourhood $U$ of $x$ and an embedding $\psi\colon \oh_X(U) \hookrightarrow \mathscr{F}(U)$. Let $\mathscr{I}\subseteq \oh_X$ be a coherent sheaf of ideals such that $\supp{\oh_X/\mathscr{I}} = X\setminus U$ (see for instance \cite[Tag 01J3]{stacks}). The proof will be concluded by the following lemma, which implies that there exists $n\ge0$ such that $\psi$ extends to a homomorphism $\varphi\colon \mathscr{I}^n \to \mathscr{F}$. Note that since $\varphi$ is generically injective and $\mathscr{I}^n$ has torsion-free modules of sections over the sheaf of domains $\oh_X$, it follows that $\varphi$ is itself injective.
\end{proof}

\begin{lem}[{\cite[Tag 01YB]{stacks}, \cite[\S V.2, Corollaire 1]{Gabriel}}]\label{lem:PowersOfI}
Let $X$ be a Noetherian scheme and $\mathscr{F} \in \qcoh_X$. Let $\mathscr{I}\subseteq \oh_X$ be a coherent sheaf of ideals, $Z = \supp{\oh_X/\mathscr{I}}$ the corresponding closed subset and $U = X\setminus Z$. Then taking sections over $U$ induces a canonical isomorphism
\[ \varinjlim_n \hom_{X}(\mathscr{I}^n, \mathscr{F}) \stackrel{\simeq}{\longrightarrow} \mathscr{F}(U). \]
\end{lem}

\subsection{The closed monoidal structure on sheaves}

Finally, we recall a few basic facts about the standard closed monoidal structure on $\Mod{\oh_X}$ and $\coh_X$. We again assume that $X$ is a locally noetherian scheme. 

We will write $\otimes = \otimes_{\oh_X}$ for the usual tensor product on $\Mod{\oh_X}$, which is simply the sheafification of the obvious tensor product of presheaves; see e.g.\ \cite[\S7.4]{GW}.

If $\mathscr{F}, \mathscr{G}$ are sheaves of $\oh_X$-modules and $x\in X$, then canonically
\[ (\mathscr{F}\otimes\mathscr{G})_x \simeq \mathscr{F}_x\otimes_{\oh_{X,x}} \mathscr{G}_x.\]
If, moreover, both $\mathscr{F}, \mathscr{G}$ are quasi-coherent (resp.\ coherent), then $\mathscr{F}\otimes\mathscr{G}$ is quasi-coherent (resp.\ coherent) and 
\[ (\mathscr{F}\otimes\mathscr{G})(U) \cong \mathscr{F}(U) \otimes_{\oh_X(U)} \mathscr{G}(U) \]
for each open affine subset $U \subseteq X$ by \cite[Corollary 7.19]{GW}.

Given $\mathscr{F}, \mathscr{G}\in\Mod{\oh_X}$, one can define the sheaf of homomorphisms $\homSh_X(\mathscr{F}, \mathscr{G}) \in \Mod{\oh_X}$ by setting
\[ \homSh_X(\mathscr{F}, \mathscr{G})(U) = \hom_U(\mathscr{F}|_U, \mathscr{G}|_U) \text{ for all open subsets } U \subseteq X. \]
The usual homomorphism group $\hom_X(\mathscr{F}, \mathscr{G})$ can be recovered as the global sections of $\homSh_X(\mathscr{F}, \mathscr{G})$.
If $\mathscr{F}$ is a coherent sheaf, then
\[ \homSh_X(\mathscr{F}, \mathscr{G})_x \simeq \hom_{\oh_{X,x}}(\mathscr{F}_x, \mathscr{G}_x) \]
canonically for each $x\in X$ by \cite[Proposition 7.27]{GW}. If, moreover, $\mathscr{G}$ is quasi-coherent (resp.\ coherent), then also $\homSh_X(\mathscr{F}, \mathscr{G})$ is quasi-coherent (resp.\ coherent) and for each open affine $U \subseteq X$, we also have
\[ \homSh_X(\mathscr{F}, \mathscr{G})(U) \simeq \hom_{\oh_{X}(U)}\left(\mathscr{F}(U), \mathscr{G}(U)\right). \]

The two construction are related by the usual adjunction
\[
\homSh_X(\mathscr{F}\otimes\mathscr{G}, \mathscr{H}) \simeq \homSh_X\left(\mathscr{F}, \homSh(\mathscr{G}, \mathscr{H})\right)
\]
for any triple $\mathscr{F}, \mathscr{G}, \mathscr{H}\in\Mod{\oh_X}$. In particular, $(\Mod{\oh_X}, \otimes, \oh_X, \homSh_X)$ is naturally a closed monoidal category, and $\coh_X$ is a closed monoidal subcategory.

\bibliography{references}

\providecommand{\bysame}{\leavevmode\hbox to3em{\hrulefill}\thinspace}
\providecommand{\MR}{\relax\ifhmode\unskip\space\fi MR }
\providecommand{\MRhref}[2]{%
  \href{http://www.ams.org/mathscinet-getitem?mr=#1}{#2}
}
\providecommand{\href}[2]{#2}
\begin{thebibliography}{AHP{\v S}T14}

\bibitem[AHHK07]{HandbookTilting}
Lidia Angeleri~H\"ugel, Dieter Happel, and Henning Krause (eds.),
  \emph{Handbook of tilting theory}, London Mathematical Society Lecture Note
  Series, vol. 332, Cambridge University Press, Cambridge, 2007.

\bibitem[AHP{\v S}T14]{St1}
Lidia Angeleri~H\"{u}gel, David Posp\'{i}\v{s}il, Jan {\v
  S}\v{t}ov\'{i}\v{c}ek, and Jan Trlifaj, \emph{{Tilting, cotilting, and
  spectra of commutative noetherian rings}}, Trans. Amer. Math. Soc.
  \textbf{366} (2014), 3487--3517.

\bibitem[AR94]{AdamekRosicky}
Ji\v{r}\'{\i} Ad\'{a}mek and Ji\v{r}\'{\i} Rosick\'{y}, \emph{Locally
  presentable and accessible categories}, London Mathematical Society Lecture
  Note Series, vol. 189, Cambridge University Press, Cambridge, 1994.

\bibitem[ARS97]{AuslanderReitenSmalo}
Maurice Auslander, Idun Reiten, and Sverre~O. Smal{\o}, \emph{Representation
  theory of {A}rtin algebras}, Cambridge Studies in Advanced Mathematics,
  vol.~36, Cambridge University Press, Cambridge, 1997, Corrected reprint of
  the 1995 original.

\bibitem[Baz03]{BazzoniPureInj}
Silvana Bazzoni, \emph{Cotilting modules are pure-injective}, Proc. Amer. Math.
  Soc. \textbf{131} (2003), no.~12, 3665--3672.

\bibitem[Be{\u{\i}}78]{Beilinson}
Alexander~A. Be{\u{\i}}linson, \emph{Coherent sheaves on $\mathbb{P}^n$ and
  problems of linear algebra}, Functional Analysis and Its Applications
  \textbf{12} (1978), no.~3, 214--216.

\bibitem[BH98]{BrunsHerzog}
Winfried Bruns and J\"{u}rgen Herzog, \emph{Cohen-macaulay rings. rev. ed},
  Cambridge Studies in Advanced Mathematics, Cambridge University Press, 1998.

\bibitem[BK03]{BuanKrause-cotilting}
Aslak~Bakke Buan and Henning Krause, \emph{Cotilting modules over tame
  hereditary algebras}, Pacific J. Math. \textbf{211} (2003), no.~1, 41--59.

\bibitem[Bor67]{Borelli}
Mario Borelli, \emph{{Some results on ampleness and divisoral schemes}},
  Pacific Journal of Mathematics \textbf{23} (1967), 217--227.

\bibitem[BvdB03]{BondalVanDenBergh}
Alexei Bondal and Michel van~den Bergh, \emph{Generators and representability
  of functors in commutative and noncommutative geometry}, Mosc. Math. J.
  \textbf{3} (2003), no.~1, 1--36, 258.

\bibitem[CB94]{CBLocFinPres}
William Crawley-Boevey, \emph{Locally finitely presented additive categories},
  Comm. Algebra \textbf{22} (1994), no.~5, 1641--1674.

\bibitem[CDT97]{Colpi}
Riccardo Colpi, Gabriella D'Este, and Alberto Tonolo, \emph{{Quasi-Tilting
  Modules and Counter Equivalences}}, Journal of Algebra \textbf{191} (1997),
  no.~2, 461--494.

\bibitem[CF07]{ColpiFuller3}
Riccardo Colpi and Kent~R. Fuller, \emph{Tilting objects in abelian categories
  and quasitilted rings}, Transactions of the American Mathematical Society
  \textbf{359} (2007), no.~2, 741--765.

\bibitem[CGM07]{ColpiHeart}
Riccardo Colpi, Enrico Gregorio, and Francesca Mantese, \emph{On the heart of a
  faithful torsion theory}, Journal of Algebra \textbf{307} (2007), no.~2,
  841--863.

\bibitem[Col99]{ColpiTiltingGrothendieck}
Riccardo Colpi, \emph{{Tilting in Grothendieck categories}}, Forum
  Mathematicum, vol.~11, Berlin; New York: De Gruyter, 1999, pp.~735--760.

\bibitem[CT95]{ColpiTrlifaj}
Riccardo Colpi and Jan Trlifaj, \emph{{Tilting Modules and Tilting Torsion
  Theories}}, Journal of Algebra \textbf{178} (1995), no.~2, 614--634.

\bibitem[CTT97]{ColpiTrlifaj2}
Riccardo Colpi, Alberto Tonolo, and Jan Trlifaj, \emph{Partial cotilting
  modules and the lattices induced by them}, Communications in Algebra
  \textbf{25} (1997), no.~10, 3225--3237.

\bibitem[Dic66]{Dickson}
Spencer~E. Dickson, \emph{{A Torsion Theory for Abelian Categories}},
  Transactions of the American Mathematical Society \textbf{121} (1966), no.~1,
  223--235.

\bibitem[EB06]{ElBashir}
Robert El~Bashir, \emph{{Covers and Directed Colimits}}, Algebras and
  Representation Theory \textbf{9} (2006), no.~5, 423--430.

\bibitem[EE05]{EnochsEstrada}
Edgar~E. Enochs and Sergio Estrada, \emph{Relative homological algebra in the
  category of quasi-coherent sheaves}, Advances in Mathematics \textbf{194}
  (2005), no.~2, 284--295.

\bibitem[Eis95]{Eisenbud}
David Eisenbud, \emph{{Commutative algebra: With a View Toward Algebraic
  Geometry}}, vol. 150, Springer-Verlag, New York, 1995.

\bibitem[FMS17]{FiorotMattielloSaorin}
Luisa Fiorot, Francesco Mattiello, and Manuel Saor\'{\i}n, \emph{Derived
  equivalences induced by nonclassical tilting objects}, Proc. Amer. Math. Soc.
  \textbf{145} (2017), no.~4, 1505--1514.

\bibitem[Gab62]{Gabriel}
Pierre Gabriel, \emph{{Des cat\'{e}gories ab\'{e}liennes}}, Bull. Soc. Math.
  Fr. \textbf{90} (1962), 323--448.

\bibitem[GT12]{G-T}
R\"{u}diger G\"{o}bel and Jan Trlifaj, \emph{{Approximations and Endomorphism
  Algebras of Modules}}, 2nd ed., vol.~1, De Gruyter, Berlin/Boston, 2012.

\bibitem[GW10]{GW}
Ulrich G\"{o}rtz and Torsten Wedhorn, \emph{{Algebraic Geometry} {I}}, Vieweg +
  Teubner, Wiesbaden, 2010.

\bibitem[Hap87]{Happel}
Dieter Happel, \emph{On the derived category of a finite-dimensional algebra},
  Commentarii mathematici Helvetici \textbf{62} (1987), 339--389.

\bibitem[Har66]{H-RD}
Robin Hartshorne, \emph{{Residues and Duality}}, Springer-Verlag, Berlin, 1966.

\bibitem[Har77]{H-AG}
\bysame, \emph{{Algebraic Geometry}}, Springer-Verlag, New York, 1977.

\bibitem[HRS96]{HRS}
Dieter Happel, Idun Reiten, and Sverre~O. Smal\o, \emph{Tilting in abelian
  categories and quasitilted algebras}, Mem. Amer. Math. Soc. \textbf{120}
  (1996), no.~575, viii+ 88.

\bibitem[H{\v S}17]{HrbekSt}
Michal Hrbek and Jan {\v S}{\v t}ov{\'\i}{\v c}ek, \emph{Tilting classes over
  commutative rings}, arXiv preprint, arXiv:1701.05534, 2017.

\bibitem[Hun76]{Hunter}
Roger~H. Hunter, \emph{Balanced subgroups of abelian groups}, Trans. Amer.
  Math. Soc. \textbf{215} (1976), 81--98.

\bibitem[Ill71]{Illusie}
Luc Illusie, \emph{Existence de r\'esolutions globales}, pp.~160--221,
  Springer, Berlin, Heidelberg, 1971.

\bibitem[JL89]{JensenLenzing}
Christian~U. Jensen and Helmut Lenzing, \emph{Model-theoretic algebra with
  particular emphasis on fields, rings, modules}, Algebra, Logic and
  Applications, vol.~2, Gordon and Breach Science Publishers, New York, 1989.

\bibitem[Kra05]{Krause}
Henning Krause, \emph{The stable derived category of a noetherian scheme},
  Compositio Mathematica \textbf{141} (2005), 423--430.

\bibitem[Len83]{LenzingTransfer}
Helmut Lenzing, \emph{Homological transfer from finitely presented to infinite
  modules}, Abelian group theory ({H}onolulu, {H}awaii, 1983), Lecture Notes in
  Math., vol. 1006, Springer, Berlin, 1983, pp.~734--761.

\bibitem[Mur07]{Murfet}
Daniel Murfet, \emph{{The Mock Homotopy Category of Projectives and
  Grothendieck Duality}}, Ph.D. thesis, Aust. National U., 2007.

\bibitem[Noo09]{Noohi}
Behrang Noohi, \emph{Explicit {HRS}-tilting}, J. Noncommut. Geom. \textbf{3}
  (2009), no.~2, 223--259.

\bibitem[NSZ19]{NicolasSaorinZvonareva}
Pedro Nicol\'{a}s, Manuel Saor\'{i}n, and Alexandra Zvonareva, \emph{Silting
  theory in triangulated categories with coproducts}, J. Pure Appl. Algebra
  \textbf{223} (2019), no.~6, 2273--2319.

\bibitem[Oda14]{odabasi}
S.~Odaba\c{s}\i, \emph{Locally torsion-free quasi-coherent sheaves}, Journal of
  Pure and Applied Algebra \textbf{218} (2014), no.~9, 1760 -- 1770.

\bibitem[Pre09]{Prest}
Mike Prest, \emph{Purity, spectra and localisation}, Encyclopedia of
  Mathematics and its Applications, vol. 121, Cambridge University Press,
  Cambridge, 2009.

\bibitem[PS15]{ParraSaorin}
Carlos~E. Parra and Manuel Saor\'{i}n, \emph{Direct limits in the heart of a
  t-structure: the case of a torsion pair}, J. Pure Appl. Algebra \textbf{219}
  (2015), no.~9, 4117--4143.

\bibitem[P{\v S}17]{PositselskiSt}
Leonid Positselski and Jan {\v S}{\v t}ov{\'\i}{\v c}ek, \emph{The
  tilting-cotilting correspondence}, arXiv preprint, arXiv:1710.02230, 2017.

\bibitem[PV18]{PsaroudakisVitoria}
Chrysostomos Psaroudakis and Jorge Vit\'{o}ria, \emph{Realisation functors in
  tilting theory}, Math. Z. \textbf{288} (2018), no.~3-4, 965--1028.

\bibitem[Ric89]{Rickard}
Jeremy Rickard, \emph{{Morita Theory for Derived Categories}}, Journal of the
  London Mathematical Society \textbf{s2-39} (1989), no.~3, 436--456.

\bibitem[{\v S}KT11]{StovicekKernerTrlifaj}
Jan {\v S}{\v t}ov{\'\i}{\v c}ek, Otto Kerner, and Jan Trlifaj, \emph{Tilting
  via torsion pairs and almost hereditary noetherian rings}, J. Pure Appl.
  Algebra \textbf{215} (2011), no.~9, 2072--2085.

\bibitem[S{\v S}19]{SlavikSt}
Alexander Sl{\'a}vik and Jan {\v S}\v{t}ov\'{i}\v{c}ek, \emph{On the existence
  of flat generators of quasicoherent sheaves}, arXiv preprint,
  arXiv:1902.05740v1, 2019.

\bibitem[{Sta}19]{stacks}
The {Stacks Project Authors}, \emph{Stacks project}, available online
  at~\url{http://stacks.math.columbia.edu}, 2019.

\bibitem[Ste75]{Stenstrom}
Bo~Stenstr\"{o}m, \emph{{Rings of Quotients}}, Springer-Verlag, New York, 1975.

\bibitem[{\v S}TH14]{St2}
Jan {\v S}\v{t}ov\'{i}\v{c}ek, Jan Trlifaj, and Dolors Herbera,
  \emph{{Cotilting modules over commutative noetherian rings}}, Journal of Pure
  and Applied Algebra \textbf{218} (2014), 1696--1711.

\bibitem[{\v S}{\v t}o06]{StPureInj}
Jan {\v S}{\v t}ov{\'\i}{\v c}ek, \emph{All {$n$}-cotilting modules are
  pure-injective}, Proc. Amer. Math. Soc. \textbf{134} (2006), no.~7,
  1891--1897.

\bibitem[{\v{S}}{\v{t}}o14]{St3}
Jan {\v{S}}{\v{t}}ov\'{i}\v{c}ek, \emph{Derived equivalences induced by big
  cotilting modules}, Advances in Mathematics \textbf{263} (2014), 45--87.

\bibitem[Tho97]{Thomason}
Robert~W. Thomason, \emph{The classification of triangulated subcategories},
  Compositio Math. \textbf{105} (1997), no.~1, 1--27.

\bibitem[TT90]{T-T}
Robert~W. Thomason and Thomas Trobaugh, \emph{{Higher Algebraic {K}-Theory of
  Schemes and of Derived Categories}}, The Grothendieck Festschrift, Progress
  in Mathematics, vol.~88, Birkh\"{a}user Boston, 1990, pp.~247--435.

\bibitem[Wak88]{Wakamatsu}
Takayoshi Wakamatsu, \emph{On modules with trivial self-extensions}, J. Algebra
  \textbf{114} (1988), no.~1, 106--114.

\end{thebibliography}
\bibliographystyle{amsalpha}

\end{document}